\documentclass[reqno]{amsart}
\input{FancyMath.sty}
\usepackage{newunicodechar}
\usepackage{standalone}
\usepackage[utf8]{inputenc}
\usetikzlibrary{arrows.meta}

\title[A ribbon partial order for links \& minimality detection via Heegaard Floer]
{
A ribbon partial order for links and \\ minimality detection via Heegaard Floer
}

\makeatletter
\@namedef{subjclassname@1991}{2020 Mathematics Subject Classification}
\makeatother

\author{\sc Gary D.~Dunkerley}
\address{Department of Mathematics, University of Georgia}
\email{garydunkerley@uga.edu}

\date{\today}


\begin{document}

\begin{abstract}
We prove that strong ribbon concordance induces a partial order on links in $S^3$, extending a theorem of Agol. Using results from knot Floer homology, we certify minimality under the ribbon partial order for a handful of knots and give the first examples of ribbon minimal knots that are not transfinitely nilpotent, resolving a question of Tagami. Using a mixture of classical techniques and recent Heegaard Floer detection results for links, we give several infinite families of links whose members are minimal under the partial order induced by strong ribbon concordance.
\end{abstract}

\maketitle

\tableofcontents

\section{Introduction}
\subsection{A partial order from strong ribbon concordance}
Two oriented links $L_0$ and $L_1$ in $S^3$ are \textit{strongly concordant} if there is a disjoint union of $n$ smooth, properly embedded annuli in $S^3 \times I$ going between them
$$
\begin{gathered}
C : \bigsqcup_{i=1}^{n} S^{1} \times I \hookrightarrow S^{3} \times I \\ C \cap (S^3\times \{0\})= L_0 \qquad C \cap (S^3 \times \{1\}) = \overline{L_1}
\end{gathered}
$$
where $\overline{L_1}$ denotes the mirror of $L_1$.
Morse theory allows one to specify embedded surfaces constructively as a union of 0-handles (minima), 1-handles (saddles), and 2-handles (maxima). 
We call $C$ a \textit{strong ribbon concordance} if it admits a handle decomposition consisting solely of  0- \& 1-handles or, equivalently, if there is a Morse function on ${S^3 \times I}$ whose restriction to the image of $C$ has no local maxima. In particular, reversing the direction of a non-trivial, strong ribbon concordance gives another strong concordance, but one that is not ribbon.
In the setting of knots, this asymmetry is reflected in many measures of complexity. If there is a ribbon concordance from ${ J }$ to ${ K }$, then...

\begin{itemize}
\item ...the Alexander polynomial for $J$ divides that of $K$. \cite{Gil84}
\item ...the Seifert genus of $K$ is greater than or equal to that of $J$. \cite{OS04a} \cite{Zem19a}
\item ...a quotient of $\pi_1(S^3 \setminus \nu K )$ contains $\pi_1(S^3 \setminus \nu J)$ as a subgroup. \cite{Gor81}
\item ...if $K$ is fibered, then so is $J$. \cite{Miy18}
\item ...the ribbon concordance induces an injection of the knot Floer and Khovanov homology of $J$ into that of $K$. \cite{Zem19a} \cite{LZ19}
\end{itemize}

We write ${ L' \leq L }$ if there exists a strong ribbon concordance from ${ L' }$ to ${ L }$. It is not hard to see that $\leq$ is a preorder on the set of knots in $S^3$. Gordon \cite{Gor81} conjectured the preorder $\leq$ is actually a partial order; this was recently proven by Agol \cite{Agol22}. We extend this result to the multicomponent case.

\begin{theorem}\label{thm:linkRibbonOrder}
Strong ribbon concordance induces a partial order ${ (\leq) }$ on oriented links in $S^3$.
\end{theorem}

\begin{remark}\label{rem:weakConcordance}
There is an alternative notion of link concordance that appears in the literature---two links are called \textit{weakly concordant} if there is a smooth, connected surface of genus-0 between them which is properly embedded in $S^{3}\times I$. One might say two links are ``\textit{weakly ribbon concordant}" if there is a weak concordance which is also a ribbon surface. Weak ribbon concordance gives a preorder on the set of links which is not antisymmetric and therefore not a partial order, see Figure \ref{fig:weakConcordanceBandDiagram}.
\end{remark}

\subsection{Ribbon minimality detection} 
We say that a link $L$ in ${ S^{3} }$ is \textit{ribbon minimal} if $L' \leq L$ implies $L' = L$.
There are many natural questions one could ask about ribbon minimal knots and links.
\begin{question}\label{quest:generalizedSliceRibbonConj}
Given a link ${ L }$, how many ribbon minima are in its smooth concordance class? In particular...
\begin{enumerate}
\item\label{itm:AtLeastOne} Does every smooth link concordance class contain a ribbon minimum?
\item\label{itm:MoreThanOne} Are there any smooth link concordance classes with two or more ribbon minima?
\end{enumerate}
\end{question}

In the knot setting, there has been substantial progress on Question \ref{quest:generalizedSliceRibbonConj}.\ref{itm:AtLeastOne}, see \cite{BS25b} and \cite{BHS26}. 
The restriction of Question \ref{quest:generalizedSliceRibbonConj}.\ref{itm:MoreThanOne} to smooth concordance classes of knots is sometimes called the \textit{Generalized Slice-Ribbon Conjecture}. It is presently unclear if the question becomes tractable even when expanded to \textit{any} strong concordance class of links, but if it were, a path to the answer likely passes through a humbler question.

\begin{question}\label{question:minima}
What knots and links are minimal with respect to the strong ribbon concordance partial order?
\end{question}

Gordon \cite{Gor81} gave the first minimality certification result for knots, proving  that knots which are \textit{transfinitely nilpotent} (Definition \ref{def:TransNilp}) and $\mathbb{Q}$\textit{-anisotropic} (Definition \ref{def:Qaniso}) are ribbon concordance minimal. 
Most previously known ribbon minimal knots are either verified as such using Gordon's criterion or confirmed by other means but conjectured to satisfy it regardless. 
Tagami classified which 2-bridge knots are ribbon minimal, yielding examples of ribbon minimal knots which satisfy one of Gordon's conditions but not the other, see Example \ref{exmp:TagamiQisotropic}.
Tagami also asked (\cite{Tag23}, Question 3.3) if there are examples of ribbon concordance minima which are not transfinitely nilpotent. We resolve Tagami's question in the affirmative.

\begin{theorem}\label{thm:minimalKnots}
The knots $Wh^{+}(T_{2,3},2)$, ${ Wh^{-}(T_{2,3},2) }$, and ${ 15n_{43522} }$ are ribbon concordance minimal.
\end{theorem}

\begin{theorem}\label{thm:notTransfinitelyNilpotent}
${ Wh^{+}(T_{2,3},2) }$ is neither ${ \mathbb{Q} }$-anisotropic nor transfinitely nilpotent. Moreover, ${ Wh^{-}(T_{2,3},2) }$ is \newline ${ \mathbb{Q} }$-anisotropic, but not transfinitely nilpotent.
\end{theorem}

\begin{figure}[htbp]
  \centering
  \begin{subfigure}{0.32\textwidth}
    \centering
    \includegraphics[width=0.8\linewidth]{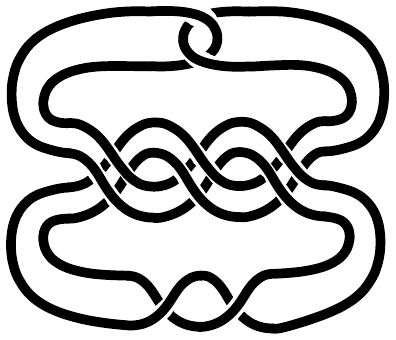}
    \caption{$Wh^+(T_{2,3},2)$}
  \end{subfigure}
  \hfill
  \begin{subfigure}{0.32\textwidth}
    \centering
    \includegraphics[width=0.8\linewidth]{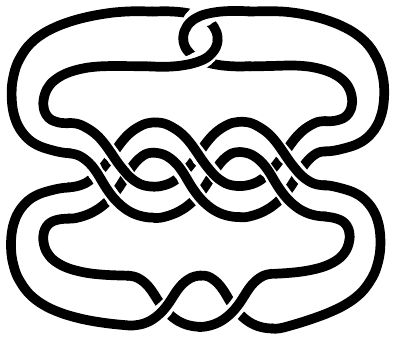}
    \caption{$Wh^-(T_{2,3},2)$}
  \end{subfigure}
  \hfill
  \begin{subfigure}{0.32\textwidth}
    \centering
    \includegraphics[width=0.8\linewidth]{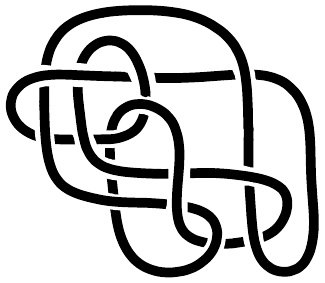}
    \caption{$15n_{43522}$}
  \end{subfigure}

  \caption{}
  \label{fig:detectedNonfiberedKnots}
\end{figure}

The proof of Theorem \ref{thm:minimalKnots} makes use of recent technical results from knot Floer homology.
Knot Floer homology, defined independently by Rasmussen \cite{Ras03} and Ozsv\'{a}th and Szab\'{o} \cite{OS04b}, is an invariant which assigns a knot $K$ in $S^3$ to a ${ (\mathbb{Z}\oplus \mathbb{Z}) }$-graded vector space ${ \widehat{\mathrm{HFK}}(K) }$.
Let $\Sigma \hookrightarrow S^{3} \times I$ be a smooth, properly embedded surface between two knots ${ K_{0} }$ and ${ K_{1} }$; knot Floer homology is functorial in the sense that it associates ${ \Sigma }$, together with a decoration, to a graded linear map ${ F_{\Sigma} : \widehat{\mathrm{HFK}}(K_{0}) \rightarrow \widehat{\mathrm{HFK}}(K_{1}) }$. \cite{Juh16, JTZ21, Zem19b}
Zemke proved \cite{Zem19a} that if there is a ribbon concordance  ${C : J \rightarrow K }$, then the functorially induced map ${F_{C} :  \widehat{\mathrm{HFK}}(J) \hookrightarrow \widehat{\mathrm{HFK}}(K) }$ is a grading-preserving inclusion.
Baldwin and Sivek \cite{BS25a} show ${ \widehat{\mathrm{HFK}} }$ detects ${ K_{0} = Wh^{+}(T_{2,3},2) }$ and detects membership in the set 
${ \left\{ K_{1} =Wh^{-}(T_{2,3},2), K_{2} = 15n_{43522}\right\} }$. \newline
For ${ K \in \{K_{0},K_{1},K_{2}\} }$, we obtain a minimality result for ${ K }$ by supposing there is some ${ J \leq K }$ and then exploiting conditions on the Alexander polynomial of $J$ and structural properties of ${ \widehat{\mathrm{HFK}}(K) }$ to force the ribbon-induced inclusion of ${ \widehat{\mathrm{HFK}}(J) }$ into it to be an isomorphism. This strategy quickly yields the result for ${ Wh^{+}(T_{2,3},2) }$ --- the proof for the remaining pair is only slightly complicated by the fact that ${ \widehat{\mathrm{HFK}} }$ only detects membership in ${ \left\{Wh^-(T_{2,3},2),15n_{43522}\right\} }$. 

\begin{remark}
In principle, a similar strategy could be used to detect ribbon minimality using other link TQFTs. 
For example, Khovanov homology also associates ribbon concordances to injective maps and there are also a handful of links detected by Khovanov homology. 
Recent work of Lobb \cite{Lobb26} explores a possible direction for a Khovanov minimality program. 
Unfortunately, one cannot na\"{i}vely implement our strategy from ${ \widehat{\mathrm{HFK}} }$ in the Khovanov setting--- Khovanov homology decategorifies to the Jones polynomial which, unlike the Alexander polynomial, is not known to have any relationship with ribbon concordance. 
\end{remark}

For links with many components, some analysis of strong ribbon concordances of split links (see Proposition \ref{prop:splitGeneric}) immediately yields
\begin{proposition}
Suppose ${ L }$ is a split disjoint union of non-split sublinks ${ L_{i} }$. Then ${ L }$ is ribbon concordance minimal if and only if each ${ L_{i} }$ is ribbon concordance minimal. 
\end{proposition}
When ${ L }$ is not split, the minimality of ${ L }$ appears to be completely independent from minimality of its sublinks. For example, every proper sublink of a Brunnian link is ribbon concordance minimal, but there are examples of Brunnian links that are strongly ribbon concordant to an unlink. On the other hand, it is reasonable to suspect there are ribbon minimal links containing components which are not ribbon minimal as knots, see Conjecture \ref{conj:W4} and the accompanying Remark \ref{rem:W4minimality} for a conjectural example. 

We prove minimality for two infinite families of links, the first of which is obtained using classical techniques and generalizes a result of Boninger and Greene.
\begin{theorem}\label{thm:BonGreeneGeneralization}
Strongly quasipositive, fibered links are strong ribbon minimal. 
\end{theorem}

\begin{corollary}
The following infinite families of links are ribbon concordance minimal.
\begin{itemize}
\item Torus links
\item Links of isolated singularities for a complex curve
\item The ``cabling" satellite pattern links ${ C_{p,q} }$
\item Special $L$-space links (see \cite{CL23}) 
\end{itemize}
\end{corollary}

There is a natural extension of knot Floer homology called \textit{link Floer homology} and denoted ${ \widehat{\mathrm{HFL}} }$. There have also lately been many link detection results using ${ \widehat{\mathrm{HFL}} }$, see \cite{BD22}, \cite{BM24}, \cite{BD24}, and \cite{Binns25}.
A version of Zemke's ribbon inclusion result also holds for ${ \widehat{\mathrm{HFL}} }$ and, imitating the strategy from the knot case using detection results from \cite{BM24}, we can give an alternative proof that 
2-component torus links are ribbon minimal.
\begin{theorem}\label{thm:HopfLinksMinimal}
Let ${ H_{n}^{-} }$ denote the 2-component torus link ${ T(2,2n) }$ with components oriented oppositely. Link Floer homology detects ribbon minimality of ${ H_{n}^{-} }$.
\end{theorem}

Recall that the $n$\textit{-twisted Whitehead links} $W_n$ are 2-component links obtained as the union of the $n^{th}$ twist knot ${ T_{n} }$ with an unknot ${ \mu }$, examples are seen in Figure \ref{fig:linkExamples}. 
Using detection results from \cite{BD24}, we obtain the following. 

\begin{theorem}\label{thm:WhiteheadLinks}
With the possible exceptions of ${ W_{4} }$ and ${\overline{W_{4}}= W_{-5} }$, the $n$-twisted Whitehead links $W_n$ are ribbon minimal.
\end{theorem}

\begin{figure}[htbp]
  \centering
  \begin{subfigure}{0.32\textwidth}
    \centering
    \includegraphics[width=0.9\linewidth]{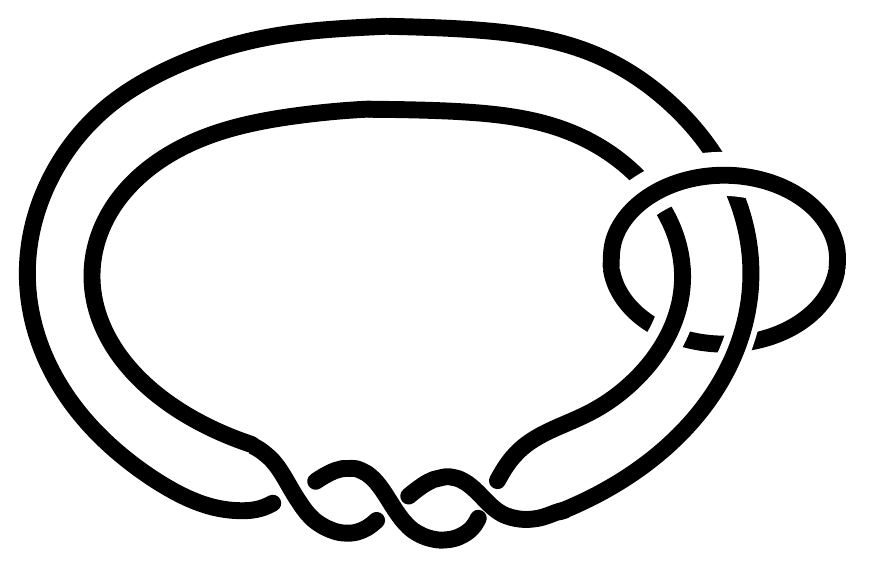}
    \caption{$C_{2,3}$}
    \label{fig:23CablingLink}
  \end{subfigure}
  \hfill
  \begin{subfigure}{0.32\textwidth}
    \centering
    \includegraphics[width=0.9\linewidth]{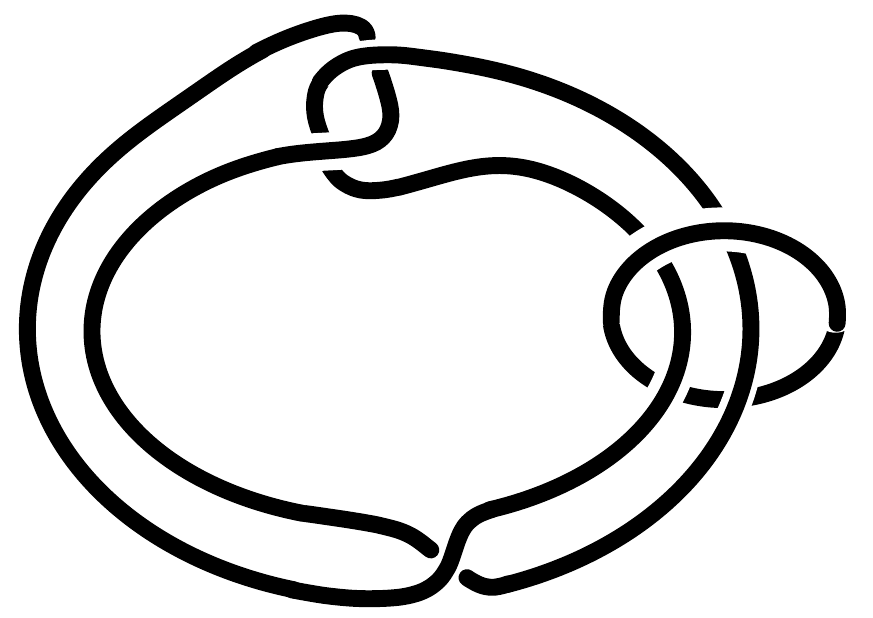}
    \caption{$W_1$}
    \label{fig:WhiteheadW1}
  \end{subfigure}
  \hfill
  \begin{subfigure}{0.32\textwidth}
    \centering
    \includegraphics[width=0.9\linewidth]{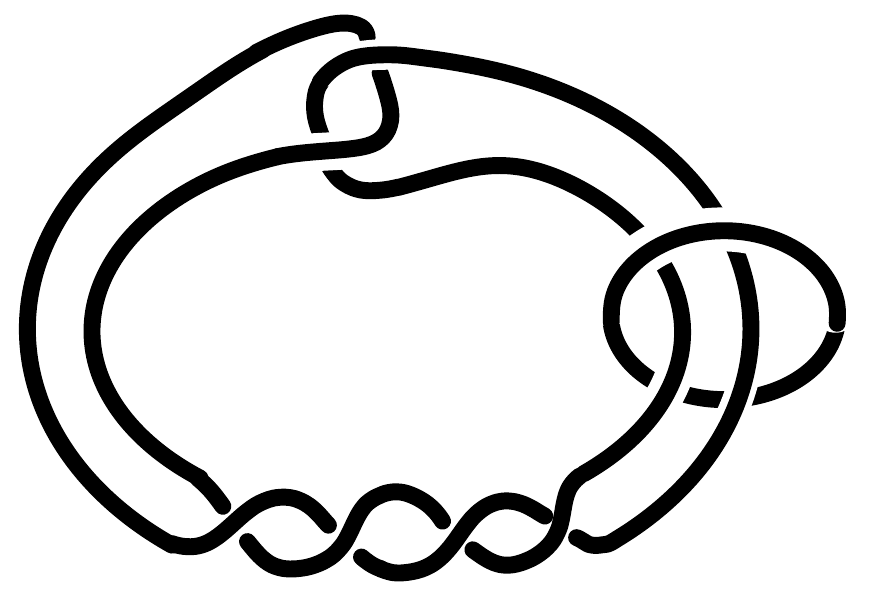}
    \caption{$W_4 = \overline{W_{-5}}$}
    \label{fig:WhiteheadW4}
  \end{subfigure}

  \caption{Figures \ref{fig:23CablingLink} and \ref{fig:WhiteheadW1} depict examples of satellite pattern links we prove are ribbon minimal. We conjecture that the link seen in Figure \ref{fig:WhiteheadW4} is also ribbon minimal.}
  \label{fig:linkExamples}
\end{figure}

\begin{conjecture}\label{conj:W4}
The link ${ W_{4} }$ is ribbon minimal. 
\end{conjecture}

The complication with ${ W_{4} }$ is that it contains the ribbon knot ${ T_{4} = 6_{1} }$ as a component. 
Although we were unable to answer the above using link Floer homology, we can impose serious constraints on a hypothetical ${ L \leq W_{4} }$ and tie Conjecture \ref{conj:W4} to a recent conjecture of Meier and Miller \cite{MM25}. 

\begin{theorem}\label{thm:componentsOfLConcordantToW4}
Suppose ${ L \leq W_{4} }$ and let ${ K }$ be the component of ${ L }$ which corresponds to ${ T_{4} \subset W_{4} }$. 
Then $L$ is not split, ${ K }$ is either the unknot or ${ T_{4} }$. Moreover, if $K=T_4$ then ${ L = W_{4} }$.
\end{theorem}

\begin{remark}\label{rem:W4minimality}
Theorem \ref{thm:componentsOfLConcordantToW4} says ${ W_{4} }$ is ribbon minimal if and only if the twist knot component ${ T_{4} \subset W_{4} }$ is not ribbon concordant to the unknot in the complement of ${ \mu \times I \subset S^{3} \times I }$. This could perhaps be answered via a careful account of how many ribbon disks the knot ${ T_{4}= 6_{1} }$ has up to isotopy rel. boundary and modulo local knotting. It is conjectured in \cite{MM25} that ${ 6_{1} }$ has finitely many distinct slice disks modulo local knotting. 
\end{remark}

\section{Acknowledgments}

We thank Fraser Binns, Jen Hom, and JungHwan Park for comments on an earlier draft and suggested references. The author also thanks his advisor, Akram Alishahi, for her support and guidance. 
The author acknowledges the assistance of ChatGPT in locating the transfinite nilpotence obstruction strategy used in the proof of Theorem~\ref{thm:notTransfinitelyNilpotent}.  
This work was partially supported by the National Science Foundation under the grants DMS-2238103 and DMS-2342252.
\section{Background}
\subsection{3-manifold topology}\label{sec:background3manifold}

Let $Y$ be an orientable 3-manifold, possibly with boundary. It is \textit{irreducible} if, given any embedding $\phi: S^2 \hookrightarrow Y$, its image bounds an embedded 3-ball. 
A surface ${ \Sigma \subset Y^{3} }$ is \textit{compressible} if there is a simple-closed curve in ${ \Sigma }$ representing a non-zero class in ${\pi_{1}( \Sigma) }$ which bounds a properly embedded disk in ${ \overline{Y^{3} \setminus \nu\Sigma } }$. Otherwise, we call the surface \textit{incompressible}. Given $\partial Y \neq \emptyset$, we say it is \textit{boundary irreducible} if $\partial Y$ is incompressible. An irreducible 3-manifold $Y$ is \textit{Haken} if it contains a properly embedded, incompressible surface with trivial normal bundle (aka ``two-sided") which is not a sphere and, if $\partial Y \neq \emptyset$, is not boundary parallel. 

Letting $Y$ and $Y'$ be 3-manifolds, a map $\phi : \pi_1(Y) \rightarrow \pi_1(Y')$ \textit{respects the peripheral structure} if for each component $F' \subset \partial Y'$, there is a component $F \subset \partial Y$ such that $\phi(i_{\ast}(\pi_1(F)))$ is contained in a subgroup conjugate to $i_{\ast}(\pi_1(F'))$. We now have all the language required to state a version of a classical theorem of Waldhausen.

\begin{theorem}[\cite{Wal68}, Corollary 6.5]\label{thm:Waldhausen}
If $M$ and $N$ are Haken 3-manifolds with boundary and ${\phi: \pi_1(N) \rightarrow \pi_1(M)}$ is an injection which respects the peripheral structure, then there is a covering map $f: N \rightarrow M$ which induces $\phi$. 
Moreover, if $\phi$ is an isomorphism respecting the peripheral structure, then there is a homeomorphism $f: N \rightarrow M$ whose induced map on ${ \pi_{1} }$ is $\phi$.
\end{theorem}

A link is a smooth, locally flat embedding of a disjoint union of circles into an orientable 3-manifold. A knot is a link with a single component. Unless otherwise specified, all links in this note are assumed to be embedded in ${ S^{3} }$. Any link in ${ S^{3} }$ bounds an orientable, connected \textit{Seifert surface} and we define
$$\chi_3(L) = \max \left\{\chi(F) \: |\; \substack{F \text{ a Seifert} \\ \text{surface for }L}\right\} $$
Gordon and Luecke showed \cite{GL89} that knots in $S^3$ are determined by the orientation-preserving homeomorphism type of their complements; this is false for links with ${ n>1 }$ components. 
However, given an orientation-preserving homeomorphism $f: S^3 \setminus \nu L_0 \rightarrow S^3 \setminus \nu L_1$ and a meridian $\mu_i$ for each component of $L_0$, if the isotopy classes of the $\mu_i$ are preserved by $f$, then the two links are isotopic. See \cite{Lac16} and references contained therein.

\subsection{Infinite cyclic covers of link complements}\label{sec:infiniteCyclicCovers}

Let ${ L = \sqcup_{i=1}^{n} L_{i} }$ be a link in ${ S^{3} }$, let ${ \mu_{i} }$ denote the meridian of ${ L_{i} }$, and let $Y := S^3 \setminus \nu L$.
Each ${ \mu_{i} }$ represents a class in ${ \pi_{1}(Y) }$; consider the group homomorphism ${ \phi : \pi_{1}(S^{3 } \setminus \nu L) \rightarrow \mathbb{Z}^{n} }$ defined by  assigning ${ [\mu_{i}] }$ to the ${ i^{th} }$ element in the standard basis of ${ \mathbb{Z}^{n} }$. The \textit{maximal abelian cover} ${ p : \widetilde{Y^{\infty}} \rightarrow  Y }$ is the covering space associated to ${ \phi }$: in particular, the group of deck transformations is ${ \mathbb{Z}^{n} }$ and 
$${ p_{\ast}(\pi_{1}(\widetilde{Y^{\infty}})) \cong \ker(\phi) \leq \pi_{1}(S^{3} \setminus \nu L) }$$The kernel of ${ \phi }$ is sometimes called the \textit{augmentation group} of ${ L }$.

\begin{definition}\label{defn:AlxPolys}

The Alexander module of a link $L$ in $S^3$ is $A_L :=  H_1(\widetilde{Y^{\infty}}, \mathbb{Z}[t_1^{\pm 1},\dots, t_n^{\pm 1}])$ and the classical multivariable Alexander polynomial $\Delta_L$ is the order ideal of $A_L$, denoted $ord(A_L)$. 
Let $TA_L$ denote the torsion submodule of $A_L$, then the \textit{torsion Alexander polynomial} is defined to be 
$$\Delta_L^{tor} := \begin{cases} ord(TA_L) & TA_L \neq 0 \\ 
1 & \text{else}\end{cases}$$
\end{definition}
We require a classical result of Torres relating the Alexander polynomial of a link to those of its sublinks.
\begin{theorem}[The Torres Condition, \cite{Tor53}, pg. 1]\label{thm:TorresCondition}
Let $L = L_1 \sqcup \cdots \sqcup L_n$ be a link with $n\geq 2$ components. Set $\ell_i = \ell k (L_i,L_n)$ and let $L' = L \setminus L_n$. 
The Alexander polynomial of $L$ satisfies the conditions 
$$\Delta_L(t_1, \cdots, t_{n-1},1) = \begin{cases}
\frac{t_1^{\ell_1}-1}{t_1-1} \Delta_{L'}(t_1) & n = 2 \\
(\prod_{i=1}^{n-1} t_i^{\ell_i} -1) \Delta_{L'}(t_1, \cdots, t_{n-1}) & n > 2

\end{cases}$$

\end{theorem}

The relationship between the (torsion) Alexander polynomial and ribbon concordance was established by Gilmer. 
\begin{theorem}[\cite{Gil84}]\label{thm:GilmersTheorem}
If $L_0 \leq L_1 $, then $\Delta^{tor}_{L_0}$ divides $\Delta^{tor}_{L_1}$.
\end{theorem}
\begin{remark}
For $K$ a knot, one has $A_K = TA_K$ and therefore $\Delta_K=\Delta_K^{tor}$. This is not the case for links in general and, when $A_L$ is not torsion, one has $\Delta_L = 0$. This happens, for example, if $L$ is split.
\end{remark}

Note that a cohomology class ${ \eta \in H^{1}(Y;\mathbb{Z}) }$ defines a group homomorphism ${ \mathbb{Z}^{n} \rightarrow \mathbb{Z} }$, so one may compose ${ \eta \circ \phi : \pi_{1}(S^{3} \setminus \nu L) \rightarrow \mathbb{Z} }$ and use this to obtain an infinite cyclic cover ${ \pi_{\eta} : \widetilde{Y^{\eta}} \rightarrow Y }$  satisfying \newline ${ (\pi_{\eta})_{\ast}\pi_{1}(\widetilde{Y^{\eta}}) = \ker(\eta\circ \phi) }$. 
In particular, ${ H_{1}(\widetilde{Y^{\eta}}, \pi_{\eta}^{-1}(p);\mathbb{Z}) }$ is a ${ \mathbb{Z}[t^{\pm 1}] }$-module where the action of ${ t }$ is induced by a generator of the group of deck transformations.

Recall that a link $L$ is \textit{fibered} if its complement is the total space of a fiber bundle over the circle 
$$F \rightarrow  S^{3} \setminus \nu L \rightarrow S^{1}$$
and each fiber ${ F }$ is a Seifert surface for ${ L }$. Letting ${ \eta }$ be a generator of ${ H^{1}(S^{1};\mathbb{Z}) }$, its pullback ${ \pi^{\ast}(\eta) \in H^{1}(S^{3}\setminus \nu L) }$ is called the \textit{fiber class} of ${ S^{3}\setminus \nu L }$. Letting ${ C : L_{0} \rightarrow L_{1} }$ be a ribbon concordance, we adopt the notation
$${ Y_{i} := S^{3} \setminus \nu L_{i} \qquad  X := (S^{3} \times I) \setminus \nu C \qquad \iota_{i} : Y_{i} \xrightarrow{\text{inclusion}} X   }$$
Choosing a cohomology class ${ \omega \in H^{1}(X;\mathbb{Z}) }$, one can likewise construct a cover ${ \pi_{\omega} :  \widetilde{X^{\omega}} \rightarrow X }$ of ${ X }$.
Letting ${ \eta_{i} := \iota_{i}^{\ast}(\omega) }$, the covers ${ \widetilde{Y_{i}^{\eta_{i}}} }$ include along maps ${ \widetilde{\iota_{i}} : \widetilde{Y_{i}^{\eta_{i}}} \rightarrow \widetilde{X^{\omega}} }$.
Recent work of Sun \cite{Sun26} tells us that if ${ C : L_{0} \rightarrow L_{1} }$ is a strong ribbon concordance, then ${ L_{1} }$ being fibered implies ${ L_{0} }$ is fibered. Moreover, he shows that there is a class ${ \omega \in H^{1}(X;\mathbb{Z}) }$ such that ${ \iota_{i}^{\ast}(\omega) }$ is the fiber class of ${ Y_{i} = S^{3}\setminus \nu L_{i} }$. 
Hence, to a strong ribbon concordance ${ C }$ of links terminating at a fibered link, there is a naturally associated infinite cyclic cover ${ \widetilde{X^{\omega}} }$ which we call the \textit{fibered cover of the concordance complement}.

We recall a version of Poincar\'{e}-Lefschetz duality for infinite cyclic covers which was described by Milnor \cite{Mil68}.

\begin{theorem}[Milnor duality]\label{thm:MilnorDuality}
Fix a field ${ \mathbb{F} }$.
Let ${ M }$ be a closed ${ n }$-manifold and ${ \widetilde{M} }$ an orientable infinite cyclic covering of ${ M }$. 
If ${ H_{\ast}(\widetilde{M}; \mathbb{F}) }$ is finitely generated then the cup product yields an orthogonal pairing
$$\smile \;:\;  H^{i}(\widetilde{M};\mathbb{F}) \otimes H^{n-i-1}(\widetilde{M}; \mathbb{F}) \rightarrow H^{n-1}(\widetilde{M};\mathbb{F}) \cong \mathbb{F}    $$
Moreover, if ${ M }$ is a compact ${ n }$-manifold with ${ \partial M \neq \emptyset }$, then we have a relative version of the above pairing
$$\smile \;:\;  H^{i}(\widetilde{M};\mathbb{F}) \otimes H^{n-i-1}(\widetilde{M}, \partial \widetilde{M}; \mathbb{F}) \rightarrow H^{n-1}(\widetilde{M},\partial \widetilde{M};\mathbb{F}) \cong \mathbb{F}    $$

\end{theorem}

\subsection{Band diagrams for ribbon surfaces}

We gave a definition of a strong concordance of links in the introduction; in this section we review the notion of a \textit{band diagram} presentation of a ribbon cobordism of links.

\begin{definition}\label{defn:bandDiagram} 
A \textit{band diagram for a ribbon surface} in $S^3 \times I$ is a triple 
$\mathcal{D} = (D, \mathbf{b}, c)$
such that 
\begin{itemize}
\item ${ D }$ is an oriented, planar link diagram for the split union of a link $L$ with the $k$-component unlink $U^k$ 
\item ${\mathbf{b} = \left\{b_{i} : I \times I \looparrowright \mathbb{R}^2 \right\} }$ is a collection of immersed rectangles with crossing decorations. Letting $A \triangle B$ denote the symmetric difference of two sets, the immersions $b_i$ satisfy properties we may summarize as: 
$$\overline{D \; \triangle \;\bigcup_{b_i \in \mathbf{b}} \mathrm{Im}(b_i|_{\partial})} \quad \text{ is an oriented link diagram.\footnote{$\triangle$ denotes the symmetric difference operation $A \triangle B := A \cup B \setminus (A \cap B)$}} $$
More explicitly, each immersion $b_i$ satisfies the following:  
\begin{itemize}
\item $b_i$ is transverse to $D$ and to itself. Transverse intersections of $b((0,1)\times I)$ with itself and with $D$ are decorated with "over" and "under" crossing data, see Figure \ref{fig:bandDiagramExamples}.
\item The images $b(I \times \{0\})$ and $b(I \times \{1\})$ (the ``attaching regions" of the 1-handle) are disjoint subarcs of $D$ whose orientation disagrees with that of $D$.
\end{itemize}

\item $c : \pi_0(D) \rightarrow \{\text{black},\text{gray}\}$ is a coloring of the components of $D$ such that 
$c^{-1}(\text{black})$ is a diagram for $L$. Similarly, $c^{-1}(\text{gray})$ is a diagram for $U^k$.
\end{itemize}
\end{definition}

It is a consequence of standard Morse theory that any band diagram encodes a properly-embedded ribbon surface in $S^3 \times I$, it is somewhat technical to see that the isotopy type of any properly-embedded ribbon surface in $S^3 \times I$ can be presented as a band diagram. We refer the interested reader to \cite{KSS82, CS93, HKM20} for an account of the ideas involved in the proof.

\begin{definition}\label{defn:MorseGraph}
The \textit{Morse graph} $\Gamma_{\mathcal{D}}$ of a band diagram $\mathcal{D} = (D,\mathbf{b},c)$ is an abstract graph defined as follows. 
The vertex set of $\Gamma_{\mathcal{D}}$ is identified with the components $C_i$ of $D$.
Two vertices $C$ and $C'$ are joined with an edge if there is a $b_i \in \mathbf{b}$ whose attaching regions lie in $C$ and $C'$ respectively.
\end{definition}

\begin{figure}[!h]
    \centering
    \begin{subfigure}[t]{0.40\linewidth}
        \centering
        \includegraphics[width=\linewidth]{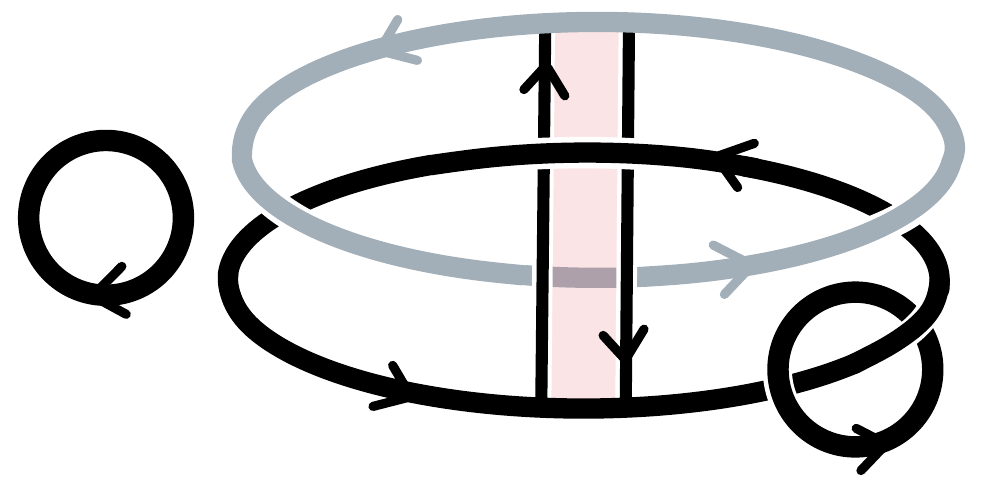}
        \caption{A band diagram representing a strong ribbon concordance between two 3-component links. Its Morse graph is $\quad \bullet\!-\!\bullet\quad\bullet\quad\bullet$.}
        \label{fig:bandDiagramExample}
    \end{subfigure}
    \hfill
    \begin{subfigure}[t]{0.48\linewidth}
        \centering
        \includegraphics[width=\linewidth]{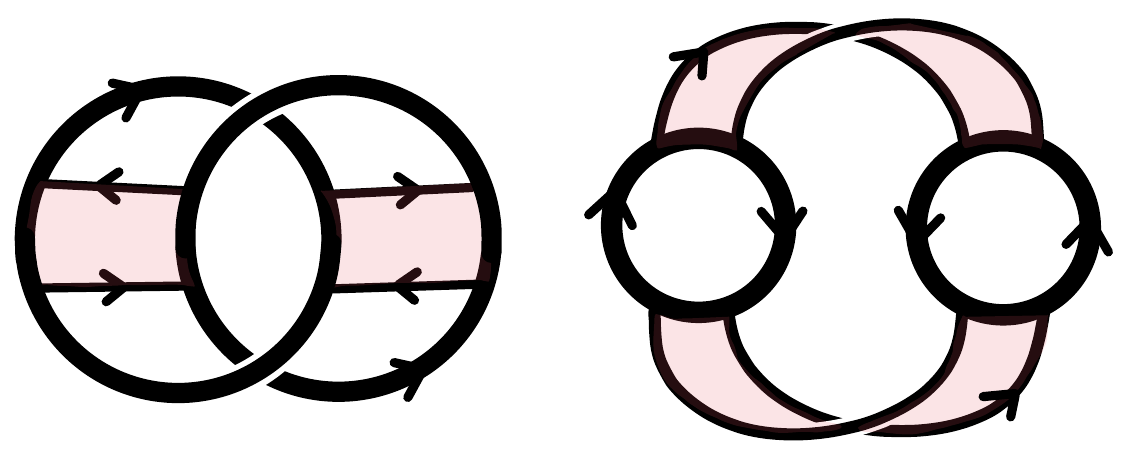}
        \caption{Two band diagrams representing properly embedded, twice-punctured annuli in $S^3 \times I$.
        The left diagram starts at the Hopf link and ends at the 2-component unlink, the right diagram goes the opposite direction. Both have Morse graph consisting of two vertices and two edges going between them.}
        \label{fig:weakConcordanceBandDiagram}
    \end{subfigure}

    \caption{}
    \label{fig:bandDiagramExamples}
\end{figure}

The following is a straightforward consequence of the definitions.
\begin{proposition}
A band diagram $\mathcal{D} = (D, \textbf{b}, c )$ represents a strong ribbon concordance if and only if ${\# \textbf{b} = \# c^{-1}(\text{gray})}$ and its associated Morse graph $\Gamma_{\mathcal{D}}$ is a forest.
\end{proposition}

\subsection{Heegaard Floer homology for knots and links}
The ``hat" knot Floer homology of a knot $K$ with coefficients in $R$ is a $\mathbb{Z}^2$-graded $R$-module 
$$\widehat{\mathrm{HFK}}(K;R) = \bigoplus\limits_{(m,a) \in \mathbb{Z}^2} \widehat{\mathrm{HFK}}_m(K,a; R).$$ 
The subscript $m$ ranges over values from the \textit{Maslov grading} $M$ and $a$ ranges over values from the \textit{Alexander grading} $A$. The graded Euler characteristic of $\widehat{\mathrm{HFK}}(K;\mathbb{Q})$ is the symmetric normalization of the Alexander polynomial: 
\begin{equation}\label{eqn:AlexanderPolyEulerChar}
\widetilde{\Delta}_K(t) = \sum\limits_{a\in \mathbb{Z}
} 
\left(\sum\limits_{m \in \mathbb{Z} 
} (-1)^m \mathrm{rank}\left(\widehat{\mathrm{HFK}}_m(K,a;\mathbb{Q})\right) \right)\cdot t^a 
\end{equation}

Let ${ L = \sqcup_{i=1}^{n} L_{i} }$ be an ${ n }$-component link in ${ S^{3} }$.  ${ L }$ can be unambiguously associated with a knot \newline ${ \kappa(L) \subset \#^{n-1}(S^{1} \times S^{2}) }$ through a process called \textit{knottification}, see \cite{OS04b}. 
The ``hat" 
knot Floer homology of ${ L }$, denoted ${ \widehat{\mathrm{HFK}}(L) }$, is knot Floer homology of its knottification inside of ${ \#^{n-1}(S^{1}\times S^{2}) }$:
$$ \widehat{\mathrm{HFK}}(L) := \widehat{\mathrm{HFK}}(\#^{n-1} (S^{1} \times S^{2}), \kappa(L)    ) $$

\begin{proposition}[\cite{Ni06}, Propositions 2.1-2]\label{prop:fiberGenusDetection}
Let ${ L }$ be an ${ n }$-component link in ${ S^{3} }$, then
\begin{enumerate}
\item The maximal supported Alexander grading of ${ \widehat{\mathrm{HFK}}(L) }$ is ${ \frac{n-\chi_{3}(L)}{2} }$
\item Moreover, ${ L }$ is fibered if and only if ${ \mathrm{rank} \left( \widehat{\mathrm{HFK}}\left( L, \frac{n-\chi_{3}(L)}{2} \right) \right) = 1 }$
\end{enumerate}
\end{proposition}

In \cite{OS08a} a related link invariant called the \textit{link Floer homology} is introduced. 
Given ${ L = \sqcup_{i=1}^{n} L_{i} }$, set
$${ \mathbb{H}(L) := \prod\limits_{i=1}^{n} \left( \mathbb{Z} + \frac{ lk(L_{i}, L \setminus L_{i})}{2} \right). }$$ 
The ``hat"
link Floer homology of $L$ with coefficients in $R$ is a $\mathbb{Z} \oplus \mathbb{H}(L)$-graded $R$-module 
$$\widehat{\mathrm{HFL}}(L;R) = \bigoplus\limits_{(m,\vec{a}) \in \mathbb{Z} \oplus \mathbb{H}(L)} \widehat{\mathrm{HFL}}_m(K,\vec{a}; R)$$the subscript $m$ iterates over values ${ \mathbb{Z} }$ from the \textit{Maslov grading} $M$ and $\vec{a}= (a_1,\cdots, a_n)$ iterates over the lattice ${ \mathbb{H}(L) }$ of \textit{Alexander multigradings} ${ (A_{1},\dots, A_{n}) }$.  
In \cite{OS08a} it is shown that the multivariable Alexander polynomial of an $n$-component link ${ L }$ in ${ S^{3} }$ relates to its link Floer homology via the graded Euler characteristic:

\begin{equation}\label{eqn:LinkFloerEulerChar}
\chi(\widehat{\mathrm{HFL}}(L)):=\sum\limits_{\vec{a} \in \mathbb{H}} \left( \sum\limits_{m \in \mathbb{Z}} (-1)^m \mathrm{rank} \left( \widehat{\mathrm{HFL}}_m(L,\vec{a}) \right) \right) \cdot \left( \prod_{i=1}^n t_i^{a_i} \right) = \prod\limits_{i =1}^n (t_i^{1/2}-t_i^{-1/2}) \cdot \Delta_L(t_1 , \cdots, t_n)
\end{equation}

One can think of ${ \widehat{\mathrm{HFK}}(L) }$ as being obtained from ${ \widehat{\mathrm{HFL}}(L) }$ by collapsing the Alexander multigrading to a single Alexander grading and then shifting the Maslov grading up by ${ \frac{n-1}{2} }$ where ${ n }$ is the number of components of ${ L }$. Note that the relationship above, combined with Proposition \ref{prop:fiberGenusDetection} implies 
\begin{corollary}
${ \widehat{\mathrm{HFL}} }$ detects  ${ \chi_{3} }$ and fiberedness. 
\end{corollary}

Zemke \cite{Zem19a} proved that ribbon concordance induces injections on all versions of knot Floer homology; the same strategy yields an analogous result for strong ribbon concordance of links in ${ S^{3} }$.

\begin{theorem}[\cite{Zem19a}]\label{thm:ribbonInjectivity}
Let ${ C : L_{0} \rightarrow L_{1} }$ be a decorated cobordism obtained from a strong ribbon concordance of links in ${ S^{3}\times I }$ such that each component of ${ C }$ is decorated with a single ${ \mathbf{w} }$ arc and a single ${ \mathbf{z} }$ arc. Then ${ \widehat{\mathrm{HFL}} }$ evaluates the decorated concordance ${ C }$ to a grading-preserving injection 
$$ F_{S^{3} \times I, C}  : \widehat{\mathrm{HFL}}(L_{0} ) \rightarrow \widehat{\mathrm{HFL}}(L_{1} )  $$
\end{theorem}

\subsection{A survey of ribbon minimality results for knots}\label{sec:surveyRibbonMinimality}

We present a survey of the literature regarding ribbon minimality for knots, beginning with the minimality criterion from \cite{Gor81}.
Recall that the \textit{lower central series} of a group $G$ is a sequence of groups $\{\gamma_i\}_{i=0}^{\infty}$ defined recursively by
$$\gamma_0(G) = G ,\qquad  \gamma_n(G) = [\gamma_{n-1}(G), G]$$
One might observe in the above that ${ \gamma_{n}(G) \leq \gamma_{n-1}(G) }$ for any ${ n \in \mathbb{N}\cup \left\{0\right\} }$, and indeed one has ${ \gamma_{n}(G) = \bigcap\limits_{i < n} \gamma_{i}(G) }$. This final observation leads to the construction of the \textit{transfinite lower central series} by defining for an arbitrary limiting ordinal ${ \lambda }$
$${ \gamma_{\lambda}(G) = \bigcap\limits_{\alpha < \lambda} \gamma_{\alpha}(G). }$$
Given ${ \alpha }$ is some ordinal, let ${ \alpha+1 }$ denote its successor ordinal and similarly define 
$$ \gamma_{\alpha+1}(G) = [\gamma_{\alpha}(G), G ].  $$ 
\begin{definition}\label{def:TransNilp}
We say that a ${ G }$ is \textit{transfinitely nilpotent} if ${ \gamma_{\alpha}(G) = \left\{1\right\} }$ for some ordinal ${ \alpha }$.
Furthermore, we say a knot $K$ in ${ S^{3} }$ is \textit{transfinitely nilpotent} if, letting ${ G = \pi_{1}(S^{3}\setminus \nu K) }$, the commutator group  ${ [G,G] }$ is transfinitely nilpotent.\footnote{Any knot group ${ G }$ is transfinitely nilpotent; this follows from the fact that every knot group abelianizes to ${ \mathbb{Z} }$ and a simple group-theoretic lemma (Lemma \ref{lem:stabilizationIfZabelianization}). Critically, ${ G }$ being transfinitely nilpotent is independent from ${ [G,G] }$ being transfinitely nilpotent.} 
\end{definition}

\begin{remark}
The notion of \textit{residual nilpotence} is occasionally mentioned in the literature; it is stronger than the notion of transfinite nilpotence. In the language of the transfinite lower central series, ${ G }$ is residually nilpotent if ${ \gamma_{\aleph_{0}}(G) = \left\{1\right\} }$ where ${ \aleph_{0} }$ is the ordinal of countable infinity.
\end{remark}

Gordon \cite{Gor81} points out that 2-bridge knots and fibered knots are transfinitely nilpotent. Moreover, the property of transfinite nilpotence is preserved under taking cables and connected sums of transfinite nilpotent knots.

\begin{definition}\label{def:Qaniso}
Given a knot $K$ in $S^3$, let $\widetilde{Y}$ denote the infinite cyclic cover of $Y := S^3 \setminus \nu K$. Let $t \in \mathrm{Aut}(H^1(\widetilde{Y};\mathbb{Q}))$ denote the automorphism induced by a generator of the group of deck transformations. One says that $K$ is $\mathbb{Q}$\textit{-anisotropic} if $H^1(\widetilde{Y};\mathbb{Q})$ contains no non-trivial $t$-invariant subspace which self-annihilates under the non-singular skew symmetric duality pairing 
$H^1(\widetilde{Y};\mathbb{Q}) \times H^1(\widetilde{Y};\mathbb{Q}) \rightarrow \mathbb{Q}$.
\end{definition}

\begin{remark}\label{rem:QanisoExamplesAndNonExamples}
Note that if $\Delta_K$ is irreducible, then it follows that $H^1(\widetilde{Y}; \mathbb{Q}) \cong Hom(\mathbb{Q}[t^{\pm 1}] / \langle \Delta_K(t) \rangle ; \mathbb{Q})$ has no proper, non-trivial submodules, so any $t$-invariant subspace is all of $H^1(\widetilde{Y};\mathbb{Q})$ and $K$ will be $\mathbb{Q}$-anisotropic by definition.
On the other hand, if $\Delta_{K}$ contains a linear factor of multiplicity one, then ${ H^{1}(\widetilde{Y};\mathbb{Q}) }$ will contain a 1-dimensional ${ t }$-invariant subspace; it is immediate from the definition that skew-symmetric bilinear forms vanish on 1-dimensional subspaces and therefore $H^{1}(\widetilde{Y};\mathbb{Q})$ is ${ \mathbb{Q} }$-isotropic.
\end{remark}

Gordon (Theorem 1.3,  \cite{Gor81}) gave the first ribbon minimality certification result, proving that if $K$ is transfinitely nilpotent and $\mathbb{Q}$-anisotropic, then $K$ is ribbon minimal. 
Until recently, all known ribbon minima either satisfied Gordon's conditions or were conjectured to satisfy them. Gordon used his criteria to show that torus knots are minimal; Boninger and Greene (\cite{BG24}, Proposition 1.7) generalizes this result to knots which are positive \& fibered and conjectured that all positive knots are $\mathbb{Q}$-anisotropic. Their proof works for the more general case of strongly quasipositive fibered knots. They also prove minimality for \textit{special alternating} knots that are fibered and have leading coefficient of $\Delta_K$ which is a prime power; this result is obtained by showing special alternating knots are $\mathbb{Q}$-anisotropic.

Baker \cite{Ba16} proved that fibered knots supporting the tight contact structure for ${ S^{3} }$ are ribbon minimal and, in particular, ${ L }$-space knots are ribbon minimal. Recently, Hom and Park noticed (\cite{HP25}, Remark 3.2) that tight, fibered knots and certain iterated cables thereof are subsumed by a more general class of ribbon minima: knots which are fibered and ``${ \gamma_{0} }$-sharp".\footnote{There is a smooth concordance invariant ${ \gamma_{0} }$ which is derived from the immersed curves reformulation of bordered Heegaard Floer homology. The ${ \gamma_{0} }$-sharp knots are those knots whose Seifert genus is detected by ${ \gamma_{0} }$.}

In later work, Boninger \cite{Bon25} shows minimality of positive knots whose leading coefficient of $\Delta_K$ is a prime power. For the proof, he leverages a result of Mayland-Murasugi to conclude that $K$ satisfying the hypothesis are transfinitely nilpotent.
He then uses results on the knot Floer homology of pseudoalternating knots and $\widehat{\mathrm{HFK}}$-injectivity to show that $J\leq K$ implies they have the same degree in their Alexander polynomial, so Lemma 3.4.ii of \cite{Gor81} implies that $J = K$.

Finally, all 2-bridge knots are residually nilpotent. Leveraging Lisca's classification of which lens spaces are linearly independent in the ${ \mathbb{Q} }$-homology cobordism group,  Tagami \cite{Tag23} showed that a 2-bridge knot is ribbon concordance minimal if and only if it is a torus knot or is not smoothly concordant to a torus knot.
In subsequent work, Abe and Tagami  \cite{AT24} completely characterize when 2-bridge knots are ribbon concordant to torus knots, explicitly determining the minimality status for all such knots.

To the best of our knowledge, 
Tagami's result yielded the first confirmed examples of ribbon concordance minima that are not $\mathbb{Q}$-anisotropic. 
\begin{example}[\cite{Tag23}]\label{exmp:TagamiQisotropic}
The 2-bridge knot 
$K=13a_{880}$ has Alexander polynomial 
$${\Delta_K = (t-2)(2t-1)(3t^2-7t + 3).}$$ Letting ${ Y := S^{3}\setminus \nu K }$ and ${ \widetilde{Y} }$ be its infinite cyclic cover, since each of the factors of $\Delta_K$ are pairwise coprime, we conclude that 
$$H_{1}(\widetilde{Y}; \mathbb{Q}) \cong \mathbb{Q}_{(2)} \oplus \mathbb{Q}_{(1/2)} \oplus \mathbb{Q}[t^{\pm 1}] / (3t^2 -7t + 3). $$

The two copies of $\mathbb{Q}$ above are $t$-invariant with the subscript indicating the Eigenvalue of the action of $t$, the contravariant ${ t }$-action likewise fixes two 1-dimensional subspaces of ${ H^{1}(\widetilde{Y};\mathbb{Q}) }$ and we conclude that $K$ is $\mathbb{Q}$-isotropic. Since $g_4(K) = 1$, Tagami's theorem implies that if ${ K }$ were not minimal then it would be ribbon concordant to a trefoil as these are the only torus knots of slice genus 1. Since $\Delta_{T_{2,3}}$ does not divide $\Delta_K$, Gilmer's theorem \cite{Gil84} implies $K$ cannot be ribbon concordant to $T_{2,3}$.
\end{example}

\section{The ribbon partial order for links}

\subsection{Strong ribbon concordance and split links}
\begin{definition}
We say that a strong concordance $C$ in $S^3 \times I$ is \textit{split} if there is some smooth, proper embedding $S : S^2 \times I \hookrightarrow S^3 \times I$ such that for all $t \in [0,1]$, $S_t = S(-,t)$ separates at least two components of $C \cap (S^3\times\{t\} )$.
\end{definition}

\begin{proposition}\label{prop:splitGeneric}
Suppose there is a strong ribbon concordance $C: L_0 \rightarrow L_1$ and $L_1$ is a split link. Then there is a split, strong ribbon concordance $C': L_0 \rightarrow L_1$. 
\begin{proof}
Fix a Morse function $f: S^{3}\times [-1,2] \rightarrow [-1,2]$ whose restriction to ${ C }$ is self-indexing. Since $C$ is ribbon, its critical values are ${ 0 }$ and ${ 1 }$ only. 
We may assume that ${ C }$ is presented by a banded unlink diagram and, in particular, we may fix some small  ${ \epsilon > 0 }$ so that for ${ A \in \left\{[-1, -\epsilon], \; [\epsilon, 1-\epsilon], [1+\epsilon,2]\right\} }$, the subspace $C \cap f^{-1}(A)$ is a product. 
We conceive of a separating sphere for ${ L_{1} }$ as a fixed embedding ${\phi : S^{2} \rightarrow f^{-1}(\left\{2\right\}) }$ which is separating for ${ C \cap f^{-1}(\left\{2\right\}) }$. We obtain a (vertical) product embedding
$$ \begin{gathered}  S  : S^{2}\times [-1,2] \rightarrow S^{3}\times [-1,2] \\ S(p, t) = S_t(p) = (\phi(p),t) \end{gathered} $$
The image of ${ S_{t} }$ splits ${ S^{3} \times \{t\} }$ into two open 3-balls, we write ${ B_{t} }$ for one of the 3-balls and assume this choice is consistent for all ${ t }$ in the sense that ${ \bigcup_{t \in [-1,2]} B_{t} }$ is connected. For brevity, we abuse notation and write ${ B_{t}^{c}  }$ to represent the interior complementary 3-ball ${ f^{-1}(\{t\}) \setminus B_{t} }$.
We prove the proposition by constructing a $C'$ which is witnessed as split by the product embedding $S \times \mathrm{Id}$.

For each of the bands $\{b_i\}_{i=1}^n$ attached at time 1 in the ribbon concordance ${ C }$, one may obtain a dual band $b_i^{\ast} \subset b_{i}$ by thickening the co-core of ${ b_{i} }$. We now think about the cobordism described by attaching these dual bands to the terminal link $L_1$ of $C$.
There are two immediately relevant consequences of ${ C }$ being a strong ribbon concordance---
\begin{enumerate}
\item\label{itm:strongRibbonConsequenceA} The result of simultaneously attaching the bands $\{b_i^{\ast}\}_{i=1}^{n}$ splits off ${ n }$ unknots $U_i$, each of which is split from the rest of the link. This is forced by ${ \chi(C) = 0 }$ and the fact that ${ C }$ has the same number of connected components as the ${ L_{i} }$.
\item\label{itm:strongRibbonConsequenceB}  Each of the ${ b_{i}^{\ast} }$ must have attaching region lying in a single component of ${ L_{1} }$. Otherwise, the fact that ${ C }$ is ribbon will force it to be a ribbon surface between links with different numbers of components, which is forbidden by the definition of a strong concordance. 
\end{enumerate}
Using item (\ref{itm:strongRibbonConsequenceA}) above, we may write $\widehat{L_{0}} = L_0' \sqcup \bigsqcup_{i=1}^n U_i$ to denote the result of attaching the ${ \left\{b_{i}^{\ast}\right\}_{i=1}^{n} }$ to ${ L_{1} }$ where the link ${ L_{0}' }$ is isotopic to ${ L_{0} }$.

Recall that we chose a splitting sphere ${ S_{2} : S^{2} \hookrightarrow f^{-1}(\left\{2\right\}) }$ for ${ L_{1} }$. Using the assumed product structure for ${ C }$ away from neighborhoods of critical level sets, the restriction to ${ t \in [1+\epsilon,2] }$ of ${ S_{t} : S^{2}\times \{t\} \hookrightarrow S^{3}\times \left\{t\right\}}$ 
has image which is disjoint from ${ C \cap f^{-1}([1+\epsilon, 2]) }$. \textit{A priori} ${ \mathrm{Im}(S_{1}) }$ may intersect the bands ${ b_{i}^{\ast} }$, but item (\ref{itm:strongRibbonConsequenceB}) tells us that the attaching region of ${ b_{i}^{\ast} }$ is either contained in ${ B_1 }$ or in ${ B_1^{c} }$ as otherwise ${ S_{1+\epsilon} }$ would not be separating for ${C \cap  f^{-1}(\{1+\epsilon\}) }$.
We bootstrap this observation and prove that the dual bands $\{b_i^{\ast}\}_{i=1}^{n}$ can be chosen so that they lie entirely in the interiors of either $B_{1}$ or $B_{1}^{c}$.

Each of the unknot components ${ U_{i} }$ is split from the rest of ${ \widehat{L_{0}}}$, so there are spanning disks $\{\Delta_i\}_{i=1}^{n}$ which are pairwise disjoint and each $\Delta_i$ is disjoint from $\widehat{L_0} \setminus U_i$.
If the attaching region for $b_i^{\ast}$ lies in $B_{1}$ but the band intersects both $B_{1}$ and $B_{1}^c$, then the spanning disk $\Delta_i$ for $U_i$ implies that $b_i^{\ast} \cap B_{1}^c$ is a union of thickened, boundary parallel arcs. In particular, we may use $\Delta_i$ to isotope $b_{i}^{\ast}$ so that it lies in the interior of $B_{1}$. It is clear that the same holds if the attaching region for $b_i^{\ast}$ lies in $B_1^c$.

\begin{remark}
It is possible that the band $b_i^{\ast}$ contains arcs belonging to two different unknot components and, hence, there are two disks that may be used to slide the band into the interior of the 3-ball containing the attaching regions of $b_i^{\ast}$. The exact choice is irrelevant to our argument. 
\end{remark}

The embeddings ${ S_{1 \pm \epsilon} : S^{2} \times \left\{1 \pm \epsilon\right\} \rightarrow f^{-1}(\{1 \pm \epsilon\}) }$ are such that  ${\mathrm{Im}( S_{1 - \epsilon}) }$ is disjoint from $\widehat{L_0} = L_0' \sqcup \bigsqcup_{i=1}^n U_i \subset f^{-1}(\{1-\epsilon\})$ and ${\mathrm{Im}( S_{1 + \epsilon}) }$ will be a splitting sphere for a link ${ L_{1}' \subset f^{-1}(\{1+\epsilon\})}$ which is isotopic to ${ L_{1} }$.
Finally, note that if a disk $\Delta_i$ intersects $\mathrm{Im}(S_{0})= \mathrm{Im}(S_{1})$, then our work above tells us it must do so in a disjoint union of closed curves, and, performing surgery along innermost disks and removing introduced closed components, we obtain a $\Delta_i'$ such that $\Delta_i' \cap \mathrm{Im}(S_{0}) = \emptyset$ and $\partial \Delta_i' = \partial \Delta_i = U_i$.

Given a component ${ K \subset L_{1} }$, we let ${ b_{K}^{\ast} }$ denote the subset of ${ \left\{b_{i}^{\ast}\right\} }$ whose attaching regions lie on ${ K }$.
$\mathrm{Im}(S_{1+\epsilon}) = \partial B_{1+\epsilon}$ is separating for the link $L_1$ by assumption, meaning there is a component $K \subset L_1 \cap B_{1+\epsilon}$ and a component $J \subset L_1 \cap B_{1+\epsilon}^{c}$. 
Our dual bands $b_i^{\ast}$ were chosen so that there are distinct components ${ K' }$ and ${ J' }$ of ${ L_{0}' }$ such that
\begin{itemize}
\item attaching bands from ${ b_{K} }$ splits $K$ into a union $K' \cup \bigsqcup_{i=1}^{\# b_{K}} U_i \subset \widehat{L_{0}} \cap B_{1-\epsilon}$ 
\item attaching bands from ${ b_{J} }$ splits $J$ into a union $J' \cup \bigsqcup_{i=1}^{\# b_{J}} U_i \subset \widehat{L_{0}} \cap B_{1-\epsilon}^{c}$ 
\end{itemize}
\noindent
Finally, if $U_i \subset B_{1-\epsilon}$, the above tells us it bounds a disk $\Delta_i' \subset B_{0}$ and likewise for ${ U_{i} \subset B_{1-\epsilon}^{c} }$. 

We now construct a ribbon concordance $C' : L_0' \to L_1'$ which is witnessed as split by ${ S_{t} }$. 
Begin with $L_0'$ at ${ t = -1 }$. At ${ t = 0 }$, birth the ${ U_{i} }$ using disks ${ \Delta_{i}' }$ to obtain the link ${ \widehat{L_{0}} }$. 
At time $t = 1$, select a co-core of each $b_i^{\ast}$ and use this as the core of the dual $(b_i^{\ast})^{\ast}$ of the dual band $b_i^{\ast}$. In particular, $(b_i^{\ast})^\ast \subset b_i^{\ast}$ and since the latter is either in the interior of $B_{1}$ or $B_{1}^c$, we conclude $(b_i^{\ast})^{\ast}$ is as well. 
The resulting link $L_1'$ will be isotopic to $L_1$.
\end{proof}
\end{proposition}

The above immediately implies the following.\footnote{Corollary \ref{cor:splitRestriction} can also be obtained from a result of Gujral and Levine. \cite{GL22} It would be interesting to know if a version of Proposition \ref{prop:splitGeneric} also holds for strong, homotopy ribbon concordances of links.}

\begin{corollary}\label{cor:splitRestriction}
If $L_0 \leq L_1$ and $L_1$ is split, then $L_0$ is split.
\end{corollary}

\subsection{Proof of Theorem \ref{thm:linkRibbonOrder}}

We make heavy use of a classical lemma of Gordon.

\begin{lemma}[Lemma 3.1 \cite{Gor81}]\label{lem:GordonsLemma}
Suppose  $C : L_0 \rightarrow L_1$ is a strong ribbon concordance, then the maps induced by inclusion of $S^3 \setminus \nu L_i$ into $(S^3 \times I) \setminus \nu C$ are such that 

$$\pi_1(S^3 \setminus \nu L_0) \hookrightarrow \pi_1((S^3 \times I) \setminus \nu C) \twoheadleftarrow  \pi_1(S^3 \setminus \nu L_1)$$

\end{lemma}
We are now ready to prove Theorem \ref{thm:linkRibbonOrder}.

\begin{proof}[Proof of Theorem \ref{thm:linkRibbonOrder}]
As with the knot case, it's clear that the relation $\leq$ induced by strong ribbon concordance is reflexive and transitive, so we need only show that it is antisymmetric.
Let $L$ and $L'$ be $n$-component links such that there are ribbon concordances 
$$C : L \rightarrow L' \qquad \text{ and } \qquad C': L' \rightarrow L$$ We will show this forces $L = L'$. In passing from the $n=1$ case to the $n>1$ case, we must handle the emergence of split links and the argument decomposes into two cases.

\vspace{5pt}

\sectionrule
\vspace{5pt}
\noindent
\textbf{Case i:} Neither ${ L }$ nor ${ L' }$ is split
\vspace{5pt}

\sectionrule 
\vspace{5pt}
\noindent
If neither link is split, one may use the same argument as in Agol's proof from the knot case. We recapitulate the argument for the reader's convenience.
The composition $C'\circ C$ is a strong ribbon concordance from $L$ to $L$ by assumption where ${ C }$ is the preimage of ${ [0, 1 / 2] }$ and ${ C' }$ is the preimage of ${ [1 / 2, 1] }$.
We adopt notation as follows, see Figure \ref{fig:schematicX} for a schematic.

$$
\begin{gathered}
\forall i \in \{0,1/2,1\}: Y_i := (S^3\times \{i\}) \setminus \nu(C' \circ C)  \\ 
X_0 := (S^3 \times [0,1/2])\setminus \nu C \qquad X_1 := (S^3 \times [1/2,1])\setminus \nu C'  \qquad X := X_0 \amalg_{Y_{1/2}} X_1 
\end{gathered}
$$

\begin{figure}[!h]
\centering
\includegraphics[width=0.55\linewidth]{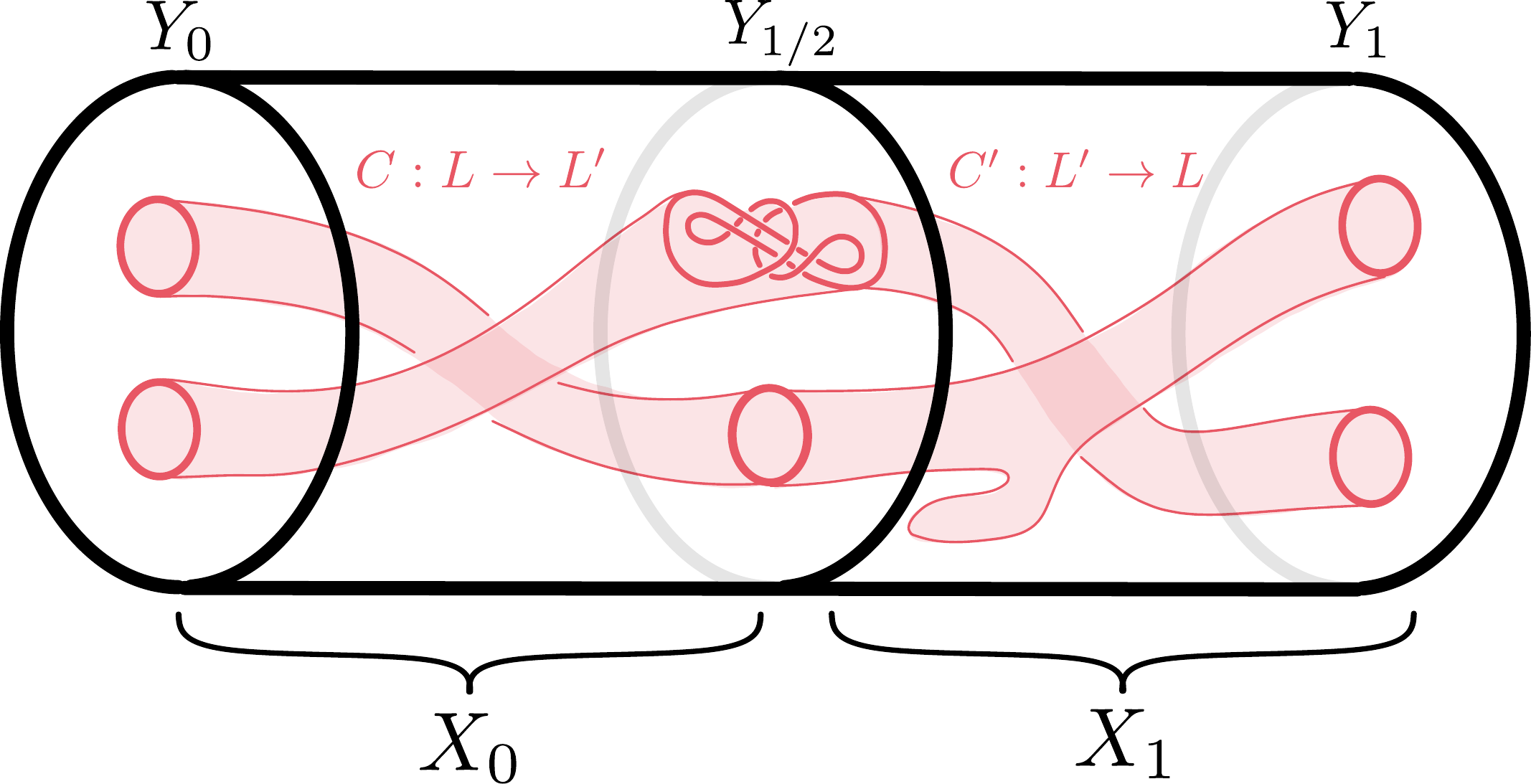}
\caption{ A schematic for the decomposition of $X$ as the union of $X_0$ and $X_1$ along the common subset $Y_{1/2}$ of their boundary strata. 
}\label{fig:schematicX}
\end{figure}

\begin{remark}
In our notation, ${ Y }$'s are 3-manifolds and ${ X}$'s are 4-manifolds; \cite{Gor81} and \cite{Agol22} use the opposite convention. 
Furthermore, we emphasize: $Y_0 \cong Y_1 \cong S^3 \setminus \nu L$ and $Y_{1/2} \cong S^3 \setminus \nu L'$. 
\end{remark}

First one shows that the inclusions $Y_0,Y_1 \hookrightarrow X$ induce a $\pi_1$-isomorphism. 
Lemma \ref{lem:GordonsLemma} tells us that the inclusion ${\iota_{1} : Y_{1} \rightarrow X }$ induces a surjective map ${ (\iota_{1})_{\ast}: \pi_{1}(Y_{1}) \twoheadrightarrow \pi_{1}(X) }$. If ${ (\iota_1)_{\ast} }$ is also injective, then one concludes that the inclusion induces an isomorphism of fundamental groups. Letting ${ R_{n}(A) }$ denote the ${ SO(n) }$-representation variety of $\pi_{1}(A)$, one can obtain the desired injection in three steps:
\begin{enumerate}
\item\label{itm:repVarietyA} 
Given ${ \rho \in R_{n}(X) }$, define ${ f(\rho) = \rho \circ (\iota_{1})_{\ast} }$ and ${ g(\rho) = \rho \circ (\iota_{0})_{\ast} }$.
One may use Lemma \ref{lem:GordonsLemma} and Proposition 2.1 of \cite{DLVW22} to obtain polynomial maps of $SO(n)$-representation varieties
$$ R_n(Y_0) \xtwoheadleftarrow{g}  R_n(X) \xhookrightarrow{\;\;f\;\;}  R_n(Y_1)$$

A careful examination of generators and relations reveals ${ f }$ and ${ g }$ are polynomial maps and, in particular, $f : R_n(X) \hookrightarrow R_n(Y_1)$ is an inclusion of a real algebraic subset.
\item\label{itm:repVarietyB}  Since $Y_0 \cong Y_1$, there is a polynomial isomorphism ${\phi : R_n(Y_0) \rightarrow R_n(Y_1)}$ and so one has a polynomial surjection ${\phi \circ g :R_n(X) \rightarrow R_n(Y_1)}$.
\item\label{itm:repVarietyC}  Since $L$ is not split, its complement is irreducible and, in fact, Haken. Hence, $\pi_1(Y_i)$ is residually finite by \cite{Hem87},  Theorem 1.1.
Given some $h \in \pi_{1}(Y_{1}) \setminus \{1\}$, the above implies that ${ \pi_{1}(Y_{1}) }$ has a finite quotient ${ p : \pi_{1}(Y_{1}) \twoheadrightarrow  H }$ in which ${ p(h)\neq 1 }$.  Cayley's Theorem implies that any finite group embeds into ${ SO(n) }$ provided ${ n \gg 0 }$; hence there is some ${ n }$ for which we have a representation ${ \rho : \pi_{1}(Y_{1}) \rightarrow SO(n)  }$ which factors through ${ p }$ such that ${ \rho(h)\neq 1 }$. 
The conclusions of (\ref{itm:repVarietyA})  and (\ref{itm:repVarietyB})  together with 
\cite{Agol22} Lemma A.2 implies ${ (\iota_{1})_{\ast} }$ induces an isomorphism
$R_n(X) \cong R_n(Y_1)$.
It follows that ${ \rho \in R_{n}(Y_{1}) }$, is the image of some ${ \rho' \in R_{n}(X) }$ under $f$ which satisfies the condition  ${ (\rho' \circ ( \iota_{1})_{\ast})(g) \neq 1 }$. Since this can be arranged for any ${ g \in \pi_{1}(Y_{1}) \setminus \{1\} }$, we conclude ${ \ker((\iota_{1})_{\ast}) }$ is trivial.
\end{enumerate}

By the above, ${ (\iota_{1})_{\ast} : \pi_{1}(Y_{1}) \rightarrow \pi_{1}(X) }$ is an isomorphism. For the next step, consider the sequence of inclusions
\begin{equation}\label{eqn:seqOfInclusions}
Y_1 \xhookrightarrow{\;\;\iota'\;\;} X_1 \xhookrightarrow{\;\;\iota''\;\;} X
\end{equation} 
By Gordon's lemma, $\iota'_{\ast} : \pi_1(Y_1) \rightarrow \pi_1(X_1)$ is surjective and, by Agol's argument sketched above, the composition ${(\iota'' \circ \iota')_{\ast} : \pi_1(Y_1) \rightarrow \pi_1(X)}$ is an isomorphism, hence $\pi_1(Y_1)\cong \pi_1(X_1)$. 
Gordon's lemma also tells us the inclusion
$$S^3 \setminus \nu L' =  Y_{1/2} \hookrightarrow X_1 = (S^3 \times I) \setminus \nu C'$$
induces an injection $(\iota_{1 / 2})_{\ast} : \pi_1(Y_{1/2}) \hookrightarrow \pi_1(X_1)\cong \pi_1(Y_1)$. We now show that this injective map respects the peripheral structure as described in \S \ref{sec:background3manifold}. In doing so, we will actually be able to show using Waldhausen's theorem (Theorem \ref{thm:Waldhausen}) that there is a covering map ${ p : Y_{1 / 2} \rightarrow Y_{1} }$ such that ${ p_{\ast} = (\iota_{1 / 2})_{\ast} }$ and, furthermore, that this map is a meridian-preserving homeomorphism.

Suppose ${ \mathcal{D} }$ is a band diagram (Definition \ref{defn:bandDiagram} ) describing the strong ribbon concordance ${ C' : L' \rightarrow L }$; the band diagram specifies a diagram ${ D' }$ for ${ L' }$. Up to applying Reidemeister moves to ${ D' }$, we can select a disk ${ \Delta }$ in the plane whose intersection with ${ \mathcal{D} }$ satisfies the following conditions:
\begin{enumerate}
\item Every component ${ D_{j}' }$ of the diagram ${ D' }$ meets ${ \Delta }$ in a single arc $a_j$
\item The arcs ${ a_{j} }$ are not nested in ${ \Delta }$
\item ${ \Delta }$ does not intersect the 0-handles or 1-handles as presented in ${ \mathcal{D} }$ (and, in particular, does not meet the attaching regions of the 1-handles of ${ \mathcal{D} }$)
\end{enumerate}

The ribbon concordance ${ C' : L' \rightarrow L }$ specifies a correspondence between the tori which comprise ${ \partial Y_{1/2} }$ and ${ \partial Y_{1} }$. 
We enumerate components of ${ \partial Y_{1 / 2} }$ as ${ F_{j}' }$ and components of ${ \partial Y_{1} }$ as ${ F_{j} }$ where we index so that ${ \nu (C') }$ connects ${ F'_{j} }$ to ${ F_{j} }$ in ${ S^{3} \times I }$. 
Letting ${ p: S^{3}\times I \rightarrow \mathbb{R}^{2} }$ be the projection whose restriction to ${ C' }$ yields the band diagram ${ \mathcal{D} }$, the images of the torus boundary ${ F_{j}' }$ under ${ p }$ inside of ${ \Delta }$ will be closed neighborhoods ${ N_{j} }$ of the ${ a_{j} }$ arcs. 
Since the ${ a_{j} }$ are not nested in the interior of ${ \Delta }$, there is a connected component ${S \subset  \Delta \setminus \cup_{j}\, int(N_{j}) }$ that intersects the boundary of each ${ N_{j} }$ and,
perturbing ${ F'_{j} }$ with an isotopy if needed, we may select a point ${ \ast_{j} \in S \cap \partial N_{j} }$ such that ${ p^{-1}(\ast_{j}) \cap F_{j}' }$ is a single point. 
We will select a basepoint ${ \ast_{\Delta} \in S }$ and a family of disjoint embedded arcs ${ \gamma_{j} \in \mathrm{Path}_{S}(\ast_{j}, \ast_{\Delta}) }$ such that ${ \gamma_{j} \pitchfork \partial N_{j} = \ast_{j} }$.

Fix a basepoint ${ \ast \in \mathbb{T}^{2} }$, let ${ \mu, \lambda }$ be loops based at ${ \ast }$ whose classes generate ${ \pi_{1}(\mathbb{T}^{2},\ast) }$. For each ${ j }$, we have inclusions
$$ 
\begin{gathered}
i_{j}' : (\mathbb{T}^{2} ,\ast) \hookrightarrow \partial Y_{1/2} \qquad \mathrm{Im}(i_{j}' ) = F_{j}'  \\
i_{j} : (\mathbb{T}^{2} , \ast ) \hookrightarrow \partial Y_{1}   \qquad \mathrm{Im}(i_{j} ) = F_{j}   
\end{gathered}
$$
such that, given ${ p : S^{3} \times I \rightarrow \mathbb{R}^{2} }$ is the projection whose restriction to ${ C' }$ yields the band diagram ${ \mathcal{D} }$, we have
$${ (p \circ  i_{j}') (\ast) = (p\circ i_{j})(\ast) = \ast_{j}. }$$
We may assume that ${ C' }$ is a product over the preimage of ${ \Delta }$ under the projection ${ p : S^{3} \times I \rightarrow \mathbb{R}^{2} }$ yielding the band diagram. We obtain various arcs by looking at the preimages of $\ast_{\Delta}$, $\ast_j$, and $\gamma_j$ under the projection ${ p }$, see Figure \ref{fig:projectionPeripheral} for a schematic.
\begin{itemize}
\item ${p^{-1}(\ast_{j})=\eta_{j} }$ is an arc living in the component of ${ \partial(\nu C') \subset X_{1} }$ that connects 
${\ast_{j} \times \left\{1 / 2\right\} \in F'_{j} \subset \partial Y_{1 / 2}  }$ to ${ \ast_{j} \times \left\{1\right\} \in F_{j} \subset \partial Y_{1} }$.
\item  ${p^{-1}( \ast_{\Delta}) = \xi }$ is an arc in ${ X_{1} }$ between the points ${\ast_{\Delta} \times \left\{1/2\right\} \in int(Y_{1 / 2}) }$ and ${ \ast_{\Delta} \times \left\{1\right\} \in int(Y_{1}) }$. 
\item  $p^{-1}(\gamma_{j}) = \Gamma_{j}$ is a disk in ${ X_{1} }$ which is bounded by the loop based at ${ \ast_{\Delta} \times \left\{1 / 2\right\} }$ obtained by concatenating our various paths:
$${ \alpha_{j}  = (\gamma_{j} \times \left\{1 / 2\right\} )^{-1}   \cdot \eta_{j}^{-1} \cdot (\gamma_{j} \times \left\{1\right\}) \cdot \xi \qquad [\alpha_{j}] \in \pi_{1}(X_{1}, \ast_{\Delta}  \times \{ 1/ 2\}) }$$
\end{itemize}

\begin{figure}[!h]
\includegraphics[width=0.4\linewidth]{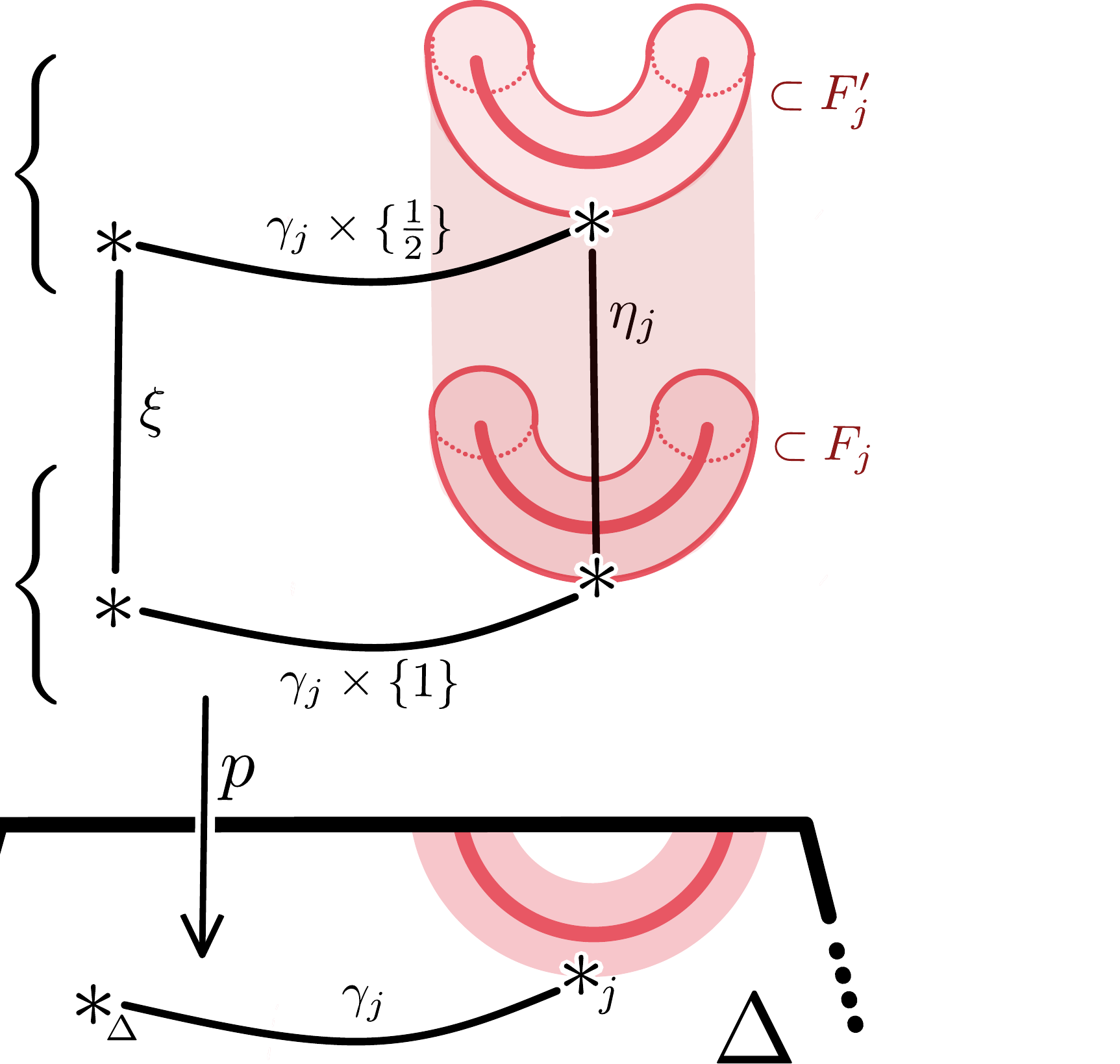}
\caption{ }\label{fig:projectionPeripheral}
\end{figure}

All of the basepoints listed above live in ${ X_{ 1} }$ and, with respect to the inclusions
$$ Y_{1 / 2} \xhookrightarrow{\;\; \iota_{1 / 2} \;\; } X_{ 1} \xhookleftarrow{\;\;\; \iota_{1} \;\;\; } Y_{1}      $$

we think of each ${ \eta_{j} }$ as being a path from ${(\iota_{1 / 2} \circ  i_{j}') (\ast) }$ to ${ (\iota_{1} \circ  i_{j})(\ast) }$ under which we obtain homotopy equivalences 
 \begin{equation}\label{eqn:homotopyEquivalences}
 \begin{gathered}
 (\iota_{1 / 2}\circ i_{j}') (\mu) \simeq \eta_{j}^{-1} \cdot (\iota_{1}\circ i_j) (\mu) \cdot \eta_{j} \\
 (\iota_{1 / 2}\circ i_{j}')(\lambda) \simeq \eta_{j}^{-1} \cdot (\iota_{1}\circ i_{j} )(\lambda) \cdot \eta_{j}        
 \end{gathered}
 \end{equation}

Using these paths, we also obtain representative images of the meridians and longitudes for ${ L }$ and ${ L' }$ as loops in ${ X_{1} }$ based at ${ \ast_{\Delta} \times \left\{1 / 2\right\} }$:
$$
\begin{alignedat}{2}
\mu_{j}' &= (\gamma_{j} \times \left\{1 / 2\right\})^{-1} \cdot  (\iota_{1 / 2}\circ i_{j}') (\mu) \cdot (\gamma_{j} \times \left\{1 / 2\right\})  \\  
\lambda_{j}' &= (\gamma_{j} \times \left\{1 / 2\right\})^{-1} \cdot  (\iota_{1 / 2}\circ i_{j}')(\lambda) \cdot (\gamma_{j} \times \left\{1 / 2\right\}) \\ 
\mu_{j} &= \xi^{-1}    \cdot (\gamma_{j} \times \left\{1 \right\})^{-1}  \cdot  (\iota_{1} \circ i_{j} )(\mu) \cdot (\gamma_{j} \times \left\{1 \right\})\cdot \xi  \\
\lambda_{j} &= \xi^{-1} \cdot (\gamma_{j} \times \left\{1 \right\})^{-1} \cdot  (\iota_{1} \circ i_{j} )(\lambda) \cdot(\gamma_{j} \times \left\{1 \right\})\cdot \xi
\end{alignedat}
$$
The classes of these loops generate $(\iota_{1/2})_{\ast}(\pi_1(Y_{1/2}))$ and $(\iota_1)_{\ast}(\pi_1(Y_1))$. Note that the existence of the disks ${ \Gamma_{j} }$ tell us that ${ [\alpha_{j}] = 1 \in \pi_{1}(X_{1}) }$. 
Combining this with the homotopy equivalences from (\ref{eqn:homotopyEquivalences}), we conclude that the $\pi_1(X_1, \ast_{\Delta} \times \{1 / 2\})$ elements satisfy.
\begin{align*}
[\mu_j] &= [\alpha_j] \cdot [\mu_{j}] \cdot [\alpha_j^{-1}]  \\ 
&= [(\gamma_{j} \times \left\{1 / 2\right\} )^{-1}   \cdot \eta_{j}^{-1}  \cdot (\iota_1 \circ i_j)(\mu) \cdot  \eta_j \cdot (\gamma_{j} \times \left\{1 / 2\right\} )]  \\ 
&= [(\gamma_{j} \times \left\{1 / 2\right\} )^{-1}   \cdot (\iota_{1 / 2} \circ i_j')(\mu) \cdot (\gamma_{j} \times \left\{1 / 2\right\} )]  \\ 
&= [\mu_{j}']  
\end{align*}
and similarly conclude ${ [\lambda_{j}'] = [\lambda_{\sigma(j)}]  }$. The above shows 
${ (\iota_{1 / 2})_{\ast} : \pi_{1}(Y_{ 1 / 2}) \rightarrow \pi_{1}(X_{1}) \cong \pi_{1}(Y_{1}) }$ respects the peripheral structure. Furthermore, since the ${ \pi_{1} }$ of a link complement is generated by classes of meridians, this likewise shows that the image of every generator of ${ \pi_{1}(Y_{1 / 2}) = \pi_{1}(S^{3} \setminus \nu L') }$ under the injective map ${ (\iota_{1 / 2})_{\ast} }$ can be identified with a generator ${ \pi_{1}(Y_{1}) = \pi_{1}(S^{3} \setminus \nu L) }$ and therefore the map is an isomorphism.

 $Y_{1/2} = S^{3} \setminus \nu L'$ and $Y_1 = S^{3}\setminus \nu L$ are complements of non-split links and are therefore Haken, so Waldhausen's Theorem (Theorem \ref{thm:Waldhausen}) gives a covering map ${ p : Y_{1 / 2} \rightarrow Y_{1}  }$ such that ${ (\iota_{ 1 / 2})_{\ast} = p_{\ast} }$. By our work above, the index of this covering map is one and therefore is a homeomorphism and, in fact, each meridian ${ \mu_{j}' }$ of ${ L' }$ is mapped under ${ p }$ to ${ \mu_{j} }$ of ${ L }$. Therefore ${ p }$ is a meridian-preserving homeomorphism between the complements of ${ L }$ and ${ L' }$ and we conclude ${ L = L' }$.

\vspace{5pt}

\sectionrule
\vspace{5pt}
\noindent
\textbf{Case ii.} One of ${ L }$ and ${ L' }$ is split.
\vspace{5pt}

\sectionrule 
\vspace{5pt}
\noindent

In light of Corollary \ref{cor:splitRestriction}, the hypothesis actually forces both ${ L }$ and ${ L' }$ to be split as otherwise, at most one of ${ C }$ or ${ C' }$ could be ribbon, yielding a contradiction. 

We change notation so that  ${ L_{0} = L }$ and ${ L_{1} = L' }$ and proceed by induction on the number of components $n$ in our links ${ L }$ and ${ L' }$. 
Ribbon concordance for knots forms the base case and this is established in \cite{Agol22}. 
For the inductive step, suppose that strong ribbon concordance is antisymmetric for $(n-1)$-component links. 
By Proposition \ref{prop:splitGeneric}, we may assume there are strong ribbon concordances 
$$C_{0 \to 1} : L_0 \to L_1 \qquad C_{1 \to 0} : L_1 \to L_0$$
which are witnessed as split by the embeddings
$$S_{0 \to 1}, S_{1 \to 0} : S^2 \times I \rightarrow S^3   \times I . $$
Moreover, the method used in the proof of Proposition \ref{prop:splitGeneric} reveals we may further assume the separating spheres for $L_1$ obtained from $S_{0\to 1}$ and $S_{1 \to 0}$ are identical.
This implies the composition 
$S_{1 \to 0} \circ S_{0 \to 1}$ is well-defined and witnesses $\Sigma := C_{1 \to 0} \circ C_{0 \to 1}$ as split. 
This means $\Sigma$ decomposes as $\Sigma' \cup (\Sigma \setminus \Sigma')$ where constituents of the union are separated by the image of $S_{1 \to 0} \circ S_{0 \to 1}$. For some $0<k<n-1$ there are $k$-component sublinks $L_i' \subsetneq L_i$ such that $\Sigma'$ is a ribbon concordance going from $L_0'$ to $L_0'$ and passes through $L_1'$. Similarly, $\Sigma \setminus \Sigma'$ goes from $L_0 \setminus L_0'$ to $L_0 \setminus L_0'$ and passes through $L_1 \setminus L_1'$.

We are done if we can show $L_0' = L_1'$ and $L_0\setminus L_0' = L_1 \setminus L_1'$.
Setting $X=(S^3 \times I)$, we let $X'$ be the (closure of) the component of $X \setminus (S^2 \times I)$ which contains $\Sigma'$ and $X \setminus X'$ be (the closure of) the component containing $\Sigma \setminus \Sigma'$. One may take $X'$ and $(X \setminus X')$ and fill in their respective $S^2 \times I$ boundary-strata with a $B^3 \times I$ to obtain two copies of $S^3 \times I$ that respectively contain the strong ribbon concordances $\Sigma'$ and $\Sigma \setminus \Sigma'$. By the inductive hypothesis, we conclude $L_0' = L_1'$ and $L_{0} \setminus L_0' = L_{1} \setminus L_1'$, and hence $L_0 = L_1$.
\end{proof}

\section{Ribbon minimality of fibered, strongly quasipositive links}\label{sec:StrongQPosFiberMinimality} 
The core of the proof of Boninger and Greene's Theorem is a pair of observations of Miyazaki.
Following a strategy suggested by Silver \cite{Sil92}, Miyazaki proved that if ${ K_{1} }$ is fibered and ${ K_{0} \leq K_{1} }$ then ${ K_{0} }$ is fibered \cite{Miy18}. 
This result was recently extended by Sun to the link case (Corollary 1.2, \cite{Sun26}). Miyazaki also noticed that if one additionally has ${ g_{3}(K_{0}) = g_{3}(K_{1}) }$, then an application of a lemma from \cite{Gor81} forces ${ K_{0} = K_{1} }$. We give a version of this latter result for the link case. 

\begin{proposition}[\cite{Gor81}, Lemmas 3.2 \& 3.4]\label{prop:fiberRibbonConcordance}
Suppose $L_1$ is fibered and non-split, ${ L_{0} \leq L_{1} }$, and ${ \chi_{3}(L_{0}) = \chi_{3}(L_{1}) }$, then ${ L_{0} = L_{1} }$.
\begin{proof}
Suppose ${ L_{0} }$ and ${ L_{1} }$ are ${ n }$-component links and  $C$ is a ribbon concordance from $L_0$ to $L_1$. We set notation 
$$Y_i = S^3 \setminus \nu L_i \qquad X = (S^3 \times I) \setminus \nu C \qquad \iota_{i}: Y_{i} \hookrightarrow X \text{ denotes }\partial\text{ inclusion} $$
As discussed in \S \ref{sec:infiniteCyclicCovers}, recent work of Sun \cite{Sun26} shows that ${ L_{0} }$ is forced to be fibered and one may construct an infinite cyclic covering ${ \widetilde{X} }$ whose boundary stratum contains the infinite cyclic fiber coverings of the ${ Y_{i} }$.
Following the contour of the proofs of Lemmas 3.2 and 3.4 from \cite{Gor81}, we study these infinite cyclic covers to obtain the result. We quickly outline the argument.
\begin{enumerate}
\item We first show that the map induced by inclusion ${ H_{1}(\widetilde{Y_{1}};\mathbb{Z}) \rightarrow H_{1} (\widetilde{X};\mathbb{Z})}$ is an isomorphism. 
This step explicitly uses the assumption that ${ L_{1} }$ is fibered and ${ \chi_{3}(L_{0}) = \chi_{3}(L_{1}) }$. 
\item Next we will show that $H_{2}(\widetilde{X};\mathbb{Z})\cong 0$. This again uses the fact that ${ L_{1} }$ is fibered. 
\item Using the first two items, we will use a theorem of Stallings to show the inclusion-induced map ${ \pi_{1}(Y_{1}) \rightarrow \pi_{1}(X) }$ is an isomorphism. 
\item We conclude the same as in the proof of Theorem \ref{thm:linkRibbonOrder}.
\end{enumerate}

\vspace{5pt}

\sectionrule
\vspace{5pt}
\noindent
\textbf{Step 1}: ${ H_{1}(\widetilde{Y_{1}};\mathbb{Z}) \xrightarrow{\cong} H_{1}(\widetilde{X};\mathbb{Z}) }$ 
\vspace{5pt}

\sectionrule 
\vspace{5pt}
\noindent
In the commutative diagram below, the horizontal maps are given by inclusion and the vertical maps come from extending scalars. Lemma \ref{lem:GordonsLemma} implies that the horizontal arrows are surjective; to prove injectivity of the upper horizontal arrow it suffices to show the left and bottom arrow are injective. 
\begin{equation}\label{eqn:diagram}  
\begin{tikzcd}  
H_1(\widetilde{Y_1}\text{; }\mathbb{Z}) \arrow[r] \arrow[d]  
& H_1(\widetilde{X}\text{; }\mathbb{Z}) \arrow[d] \\  
H_1(\widetilde{Y_1}\text{; }\mathbb{Q}) \arrow[r]  
& H_1(\widetilde{X}\text{; }\mathbb{Q})  
\end{tikzcd}  
\end{equation}

For the leftmost vertical arrow, since ${Y_{1} = S^{3} \setminus \nu L_{1} }$ is fibered, there's a Seifert surface $F$ of ${ L_{1} }$ such that ${ Y_{1} }$ is the mapping torus for some diffeomorphism of ${ F }$. This means $\widetilde{Y_1} \cong  F \times \mathbb{R}\simeq F$ and, since ${ F }$ has non-empty boundary it follows that ${ H_{1}(\widetilde{Y_{1}};\mathbb{Z}) }$ is a free abelian group, in particular it has no ${ \mathbb{Z} }$-torsion. It follows that the vertical map ${ H_{1}(\widetilde{Y_{1}};\mathbb{Z}) \rightarrow H_{1}(\widetilde{Y}_{1}; \mathbb{Z})\otimes_{\mathbb{Z}} \mathbb{Q} \cong  H_{1}(\widetilde{Y}_{1}; \mathbb{Q}) }$ on the left-hand side of Equation \ref{eqn:diagram} is injective.

For the bottom arrow, consider the portion of the long exact sequence of the pair
\begin{equation}\label{eqn:LES_Milnor_duality} 
H^{1} (\widetilde{X};\mathbb{Q}) \xrightarrow{i^{\ast} } H^{1} (\partial \widetilde{X};\mathbb{Q}) \xrightarrow{\delta} H^{2} (\widetilde{X}, \partial \widetilde{X}; \mathbb{Q})  
\end{equation}
Milnor's duality theorem for infinite cyclic covers (Theorem \ref{thm:MilnorDuality}) equips ${ H^{1}(\widetilde{X};\mathbb{Q}) }$ and ${ H^{1}(\partial \widetilde{X};\mathbb{Q}) }$ with . 
As described in \cite{Mil68} pg. 130, the maps ${ i^{\ast} }$ and ${ \delta }$ are adjoint with respect to this form\footnote{One has for any classes ${ a \in H^{1}(\widetilde{X};\mathbb{Q}) }$ and ${ b \in H^{1}(\partial \widetilde{X}; \mathbb{Q}) }$ that ${ \langle i^{\ast}a , b \rangle_{\partial \widetilde{X}} = \langle a , \delta b \rangle_{(\widetilde{X},\partial \widetilde{X})}  }$ where the left bilinear form is given by Milnor duality of infinite cyclic covers without boundary and the right is given by relative Milnor duality.} and, in particular, we may identify ${ \mathrm{ker}\, \delta }$ with the orthogonal complement of ${ \mathrm{Im}\, i^{\ast} }$. Exactness of (\ref{eqn:LES_Milnor_duality}) implies ${ \mathrm{Im}\, i^{\ast} }$ is equal to its own orthogonal complement and is therefore a half-dimensional subspace of ${ H^{1}(\partial \widetilde{X}; \mathbb{Q}) }$ on which the Milnor duality pairing vanishes. Since vector spaces have the same dimension as their duals, we conclude
$${ \dim H_{1}(\widetilde{X};\mathbb{Q}) = \frac{1}{2} \dim H_{1}(\partial \widetilde{X};\mathbb{Q})  }$$
However, we have ${ H_{1}(\partial \widetilde{X};\mathbb{Q}) = H_{1}(\widetilde{Y}_{0}\sqcup \widetilde{Y}_{1};\mathbb{Q}) \cong H_{1}(\widetilde{Y_{0}};\mathbb{Q}) \oplus H_{1}(\widetilde{Y_{1}};\mathbb{Q}) }$, which tells us 

\begin{equation}\label{eqn:dimensionAdditivity}
\dim H_{1}(\widetilde{X} ; \mathbb{Q}) = \frac{1}{2} \bigg( \dim H_{1}(\widetilde{Y_{0} };\mathbb{Q}) + \dim H_{1}( \widetilde{Y_{1} };\mathbb{Q})  \bigg)   
\end{equation}

Recall that ${ \chi_{3}(L_{0}) = \chi_{3}(L_{1})  }$ by assumption and that the covers ${ \widetilde{Y_{i}} }$ are exactly the coverings induced by the fiber classes ${ \iota_{i}^{\ast}(\omega) }$.
Our assumptions together with a theorem of Gabai \cite{Gab86} implies
$$ \dim H_{1}(\widetilde{Y_{0} }; \mathbb{Q}) 
= \dim H_{1}(\widetilde{Y_{1} }; \mathbb{Q})      
$$
 Applying the above to Equation \ref{eqn:dimensionAdditivity}, we have
 ${  \dim_{\mathbb{Q}}(H_{1}(\widetilde{X})) = \dim_{\mathbb{Q}}(H_{1}(\widetilde{Y_{1}}))}$, hence the horizontal map \newline ${ H_{1}(\widetilde{Y_{1}};\mathbb{Q}) \rightarrow H_{1}(\widetilde{X};\mathbb{Q}) }$ from Equation \ref{eqn:diagram} is injective by the pigeonhole principle.

\vspace{5pt}

\sectionrule
\vspace{5pt}
\noindent
\textbf{Step 2}: ${ H_{2}(\widetilde{X};\mathbb{Z})\cong 0 }$
\vspace{5pt}

\sectionrule 
\vspace{5pt}
\noindent

By Lemma \ref{lem:GordonsLemma} the inclusion ${\iota_{1} : Y_{1} \hookrightarrow X }$ induces a surjection ${ \pi_{1}(Y_{1}) \twoheadrightarrow \pi_{1}(X) }$; we abelianize and conclude the inclusion-induced maps ${ H_{1}(\widetilde{Y_{1}};R) \rightarrow H_{1}(\widetilde{X};R) }$ and  ${ H_{1}(\partial \widetilde{X}; R) \rightarrow H_{1}(\widetilde{X}; R) }$ are likewise surjective for any coefficient ring ${ R }$. The long exact sequence of ${ (\widetilde{X},\partial \widetilde{X}) }$ can be used to see ${ H_{1}(\widetilde{X}, \partial \widetilde{X}) \cong 0 }$.\footnote{Note that the boundary stratum of $X$ is connected; it is equal to $Y_0\sqcup Y_1 \sqcup \bigsqcup_{i=1}^n S^1\times D^2\times I$.

}

Let ${ R }$ be a field. Relative Milnor duality implies 
${ H^{2}(\widetilde{X};R) \cong \mathrm{Hom}_{R}\big(H^{1}(\widetilde{X}, \partial \widetilde{X};R),R\big) \cong H_{1}(\widetilde{X}, \partial \widetilde{X}; R) }$ and, dualizing again, one obtains
$${ 0 \cong \mathrm{Hom}\big(H_{1}(\widetilde{X}, \partial \widetilde{X}; R), R \big) 
\cong  \mathrm{Hom}\big(H^{2}(\widetilde{X}; R), R \big) \cong 
H_{2}(\widetilde{X};R) 
}$$
and the universal coefficient theorem implies ${ H_{2}(\widetilde{X};\mathbb{Z}) \otimes_{\mathbb{Z}} R \cong 0 }$. Setting ${ R = \mathbb{Q}}$, the above implies ${ H_{2}(\widetilde{X};\mathbb{Z}) }$ is ${ \mathbb{Z} }$-torsion. Setting ${ R = \mathbb{Z} / p\mathbb{Z} }$ for ${ p }$ a prime, one similarly concludes that the ``multiplication-by-${ p }$" map \newline
${ \times p :  H_{2}(\widetilde{X};\mathbb{Z}) \rightarrow H_{2}(\widetilde{X};\mathbb{Z})  }$ is surjective.

Recall that ${ H_{2}(\widetilde{X};\mathbb{Z}) }$ is a finitely generated module over the Noetherian ring ${ \mathbb{Z}[t^{\pm 1}] }$.
A surjective endomorphism of a finitely-generated module over a Noetherian ring is injective. Hence the map ${\times p}$ has trivial kernel for any prime ${ p }$ and we conclude ${ H_{2}(\widetilde{X};\mathbb{Z})\otimes_{\mathbb{Z}} \mathbb{Z} / p \mathbb{Z} \cong 0 }$ for all ${ p }$. Therefore, ${ H_{2}(\widetilde{X};\mathbb{Z}) \cong 0 }$.

\vspace{5pt}

\sectionrule
\vspace{5pt}
\noindent
\textbf{Step 3}: ${ \pi_{1}(\widetilde{Y_{1}}) \xrightarrow{\cong} \pi_{1}(\widetilde{X}) }$
\vspace{5pt}

\sectionrule 
\vspace{5pt}
\noindent

Let ${ G \cong \pi_{1}(\widetilde{Y_{1}}) }$ and ${ H \cong \pi_{1}(\widetilde{X}) }$. We adopt the notation ${ G_{\alpha} }$ and ${ H_{\alpha} }$ to denote the ${ \alpha^{th} }$ term in their lower central series. Note that since ${ G \cong \pi_{1}(F\times \mathbb{R}) \cong \pi_{1}(F) }$ is free, it follows that it's transfinitely nilpotent. 
By Step 1, ${ H_{1}(\widetilde{Y_{1}};\mathbb{Z}) \rightarrow H_{1}(\widetilde{X};\mathbb{Z}) }$ is an isomorphism and ${ H_{2}(\widetilde{X};\mathbb{Z}) \cong 0 }$, one would like to use a theorem of Stallings \cite{Sta65} to conclude that the maps 
${ \pi_{1}(\widetilde{Y_{1}}) \cong G \rightarrow H \cong  \pi_{1}(\widetilde{X}) }$ induce injections ${ G / G_{\alpha} \rightarrow H / H_{\alpha}  }$ for all ${ \alpha }$. The transfinite nilpotence of ${ G }$ implies ${ \pi_{1}(\widetilde{Y}_{1}) \rightarrow \pi_{1}(\widetilde{X}) }$ is injective. Using Gordon's lemma, one knows that ${ \pi_{1}(\widetilde{Y_{1}}) \rightarrow \pi_{1}(\widetilde{X}) }$ is surjective and therefore is an isomorphism. One may apply the five lemma to the diagram below to conclude that the inclusion-induced map ${ \pi_{1}(Y_{1}) \rightarrow \pi_{1}(X) }$ is an isomorphism.

\[\begin{tikzcd} 1 & {\pi_1(\widetilde{Y_1})} & {\pi_1(Y_1)} & {\mathbb{Z}} & 1 \\ 1 & {\pi_1(\widetilde{X})} & {\pi_1(X)} & {\mathbb{Z}} & 1 \arrow[from=1-1, to=1-2] \arrow["\cong"', from=1-1, to=2-1] \arrow[from=1-2, to=1-3] \arrow["\cong"', from=1-2, to=2-2] \arrow["{\omega \circ \phi}", from=1-3, to=1-4] \arrow["{(\therefore \;\cong)}"', from=1-3, to=2-3] \arrow[from=1-4, to=1-5] \arrow["{\cong }"', from=1-4, to=2-4] \arrow["{\cong }"', from=1-5, to=2-5] \arrow[from=2-1, to=2-2] \arrow[from=2-2, to=2-3] \arrow["{\omega \circ \phi}", from=2-3, to=2-4] \arrow[from=2-4, to=2-5] \end{tikzcd}\]

\vspace{5pt}

\sectionrule
\vspace{5pt}
\noindent
\textbf{Step 4}: ${ L_{0} = L_{1} }$
\vspace{5pt}

\sectionrule 
\vspace{5pt}
\noindent
The final step is identical to the final step for the non-split case in the proof of Theorem \ref{thm:linkRibbonOrder}. That is, we argue that the injective map ${ \pi_{1}(Y_{0}) \rightarrow \pi_{1}(X) \cong \pi_{1}(Y_{1}) }$ is an isomorphism respecting the peripheral structure and, by Waldhausen's theorem, obtain a homeomorphism between ${ Y_{0} }$ and ${ Y_{1} }$ which preserves meridians. Hence, ${ L_{0} = L_{1} }$.
\end{proof}
\end{proposition}

Recall the analogue of the ``smooth slice genus" of a link:
$$ \chi_4(L) = \max \left\{\chi(F) \: \middle| \; \substack{F \hookrightarrow D^4 \text{ is a smooth, connected} \\ \text{  properly-embedded, \& orientable} \\ \text{surface  with } \partial F = L}\right\} $$

\begin{theoremrep}{thm:BonGreeneGeneralization}
Fibered, strongly quasipositive links are ribbon minimal.
\begin{proof}
The proof is identical to that presented in \cite{BG24} only we work with the Euler characteristic for convenience.
Let $L_{1}$ be fibered and strongly quasipositive and suppose $L_{0} \leq L_{1}$. 
First, it is clear that $\chi_4(L_{0}) \geq  \chi_3(L_{0})$. Next, ${ L_{0} }$ and ${ L_{1} }$ are smoothly concordant, so $\chi_4(L_{0}) = \chi_4(L_{1})$. Results of Ni (Proposition \ref{prop:fiberGenusDetection}) and ribbon injectivity of ${ \widehat{\mathrm{HFL}} }$ (Theorem \ref{thm:ribbonInjectivity}) also imply ${ \chi_{3}(L_{0}) \geq \chi_{3}(L_{1}) }$. 
Finally, ${ L_{1} }$ is strongly quasipositive and work of Rudolph \cite{Rud93} implies ${ \chi_{3}(L_{1}) = \chi_{4}(L_{1}) }$. 
We therefore have
$$
\chi_4(L_1)= \chi_4(L_0) \geq \chi_3(L_0) \geq \chi_3(L_1) = \chi_4(L_1) \qquad \Rightarrow \qquad \chi_3(L_0) = \chi_3(L_1)
$$
Hence ${ L_{1} }$ is fibered and ${ \chi_{3}(L_{0}) = \chi_{3}(L_{1}) }$, so we are done by Proposition \ref{prop:fiberRibbonConcordance}. 
\end{proof}
\end{theoremrep}

\section{Ribbon minima via Heegaard Floer homology}

\subsection{Knot Floer homology \& ribbon minima}\label{sec:minimaDetectionViaKnotFloer}
Let ${ K }$ be a knot in ${ S^{3} }$.
Given a homogeneous element $x \in \widehat{\mathrm{HFK}}(K)$, its $\delta$-grading is $\delta(x) := A(x)-M(x)$. 
A knot $K$ is called \textit{Floer-thin} if $\widehat{\mathrm{HFK}}(K)$ is supported in a single $\delta$ grading. 
We warm up with a toy example. 

\begin{proposition}\label{prop:thinDetectedMinima}
Suppose $K \subset S^3$ is Floer-thin, has irreducible Alexander polynomial, and is not smoothly slice. If $\widehat{\mathrm{HFK}}$ detects $K$, then $K$ is a ribbon concordance minima.

\begin{proof}
Suppose $K$ satisfies the above conditions and $J \leq K$. 
By \cite{Gil84}, $\Delta_J$ divides $\Delta_K$ and since $\Delta_K$ is irreducible, either $\Delta_J  \in \{1,\Delta_K\}$.
Since $K$ is Floer-thin, the grading-preserving injection
$\widehat{\mathrm{HFK}}(J) \hookrightarrow \widehat{\mathrm{HFK}}(K)$ given by \cite{Zem19a} forces $J$ to be Floer thin. If $\Delta_J = 1$, then $\mathrm{rank}(\widehat{\mathrm{HFK}}(J))=1$ and is concentrated in Alexander grading $0$ and, by \cite{OS04a}, it follows $J$ is the unknot. 
Hence, $\Delta_J = \Delta_K$ and $K$ and $J$ Floer-thin forces
$\widehat{\mathrm{HFK}}(K) \cong \widehat{\mathrm{HFK}}(J)$.
\end{proof}
\end{proposition}

$\widehat{\mathrm{HFK}}$-detection results presently account for all knots of $\leq 5$ crossings--- $\widehat{\mathrm{HFK}}$ detects the unknot. \cite{OS04a} It also detects $3_1$, $4_1$ \cite{Ghi08}, $5_1$ \cite{FBW24}, and $5_2$ \cite{BS25a}, all of which are smoothly non-slice, Floer thin, and have irreducible Alexander polynomial. Hence, we obtain   
\begin{corollary}
$\widehat{\mathrm{HFK}}$ certifies ribbon minimality for all
knots with crossing number $\leq 5$. 
\end{corollary}
We now turn to the proof of Theorem \ref{thm:minimalKnots} which we divide into proofs of Propositions \ref{prop:Wh+} and \ref{prop:Wh^-and15}.
 
\begin{proposition}\label{prop:Wh+}
$Wh^+(T_{2,3},2)$ is ribbon concordance minimal.
\begin{proof}
Letting ${ K=Wh^+(T_{2,3},2) }$, the factors of $\Delta_K$ are linear and therefore cannot be Alexander polynomials for any knot, so $J \leq K$ implies $\Delta_J \in \{1, \Delta_K\}$. 
Where subscripts indicate Maslov gradings, the knot Floer homology of $K$ decomposed along the Alexander grading is 
\begin{equation}\label{eqn:BaldwinSivek_knotFloerComputation}
\widehat{\mathrm{HFK}}(K,a;\mathbb{Q}) \cong \begin{cases}
\mathbb{Q}^2_{(-1)} & a = 1 \\
\mathbb{Q}^4_{(-2)} \oplus \mathbb{Q}_{(0)} & a = 0 \\
\mathbb{Q}^2_{(-3)} & a=-1
\end{cases}
\end{equation}

The Euler characteristic of knot Floer homology over $\mathbb{Q}$ yields the (symmetric normalization of the) Alexander polynomial:
\begin{equation}\label{eqn:AlexanderPolyEulerCharWh+}
\sum\limits_{a\in \mathbb{Z}
} 
\left(\sum\limits_{m \in \mathbb{Z} 
} (-1)^m \mathrm{rank}\left(\widehat{\mathrm{HFK}}_m(K,a;\mathbb{Q})\right) \right)\cdot t^a = -2t^{-1} + 5 -2t
\end{equation}

Comparing Equations \ref{eqn:BaldwinSivek_knotFloerComputation} and \ref{eqn:AlexanderPolyEulerCharWh+}, one might notice for this example that within the summands indexed by the Alexander grading there is no internal cancellation from the alternating signs. In particular, since $J\leq K$ \cite{Zem19a} tells us $\widehat{\mathrm{HFK}}(J)$ is a subspace of $\widehat{\mathrm{HFK}}(K)$ and the only way one could have $J \leq K$ and $\Delta_J=1$ is if $\mathrm{rank}(\widehat{\mathrm{HFK}}(J))=1$ and concentrated in Alexander grading 0. By \cite{OS04a} this forces $J$ to be the unknot and implies $K$ is smoothly slice but $s(K)= 2$ so this is impossible.\footnote{We remark that $Wh^+(T_{2,3},2)$ is one of the examples used in \cite{HO08} to show that $\tau \neq \frac{s}{2}$. } 
Hence, we must have $\Delta_J = \Delta_K$ and, again, since there are no canceling terms in the graded Euler characteristic computation, this means the grading-preserving inclusion of knot Floer homology forces $\widehat{\mathrm{HFK}}(K;\mathbb{Q})\cong \widehat{\mathrm{HFK}}(J;\mathbb{Q})$. Finally, \cite{BS25a} tells us $J=K$ and concludes the proof.
\end{proof}
\end{proposition}

\begin{proposition}\label{prop:Wh^-and15}
The knots $Wh^-(T_{2,3},2)$ and $15n_{43522}$ are ribbon concordance minima.
\begin{proof}
Baldwin and Sivek proved \cite{BS25a} that $\widehat{\mathrm{HFK}}$ detects membership of
$$\{K_1=Wh^-(T_{2,3},2),\; K_2=15n_{43522} \}$$

The ${ K_{i} }$ have symmetric Alexander polynomial
$\widetilde{\Delta}_{K_i}=2t-3+2t^{-1}$ and Fox-Milnor obstructs them from being topologically slice--- the requisite $f(t) = at+b$ yields an impossible equality of $\mathbb{Z}[t^{\pm 1}]$-elements
$$f(t)f(t^{-1}) = ab(t+t^{-1})+(a^2+b^2) = 2(t+t^{-1})+3.$$
Additionally, note that $t \cdot \widetilde{\Delta}_{K_i}$ is irreducible in $\mathbb{Z}[t]$, so given some $J \leq K_i$, Gilmer's theorem and the Fox-Milnor obstruction imply $\Delta_J = \Delta_K$. \textit{A priori} it's possible that $K_1 \leq K_2$ or vice versa. The $K_i$ have the same $s$ and $\tau$ invariant, so one requires a subtler concordance invariant to obtain an obstruction.
Using \texttt{KnotJob} \cite{KnotJob} we are able to show
$$s_{Sq^1,-}^{odd}(K_1 \# \overline{K_2}) = -2$$
where $s_{Sq^1,-}^{odd}$ is a smooth concordance invariant that vanishes on smoothly slice knots, see \cite{DLS26}.
Hence $K_1$ and $K_2$ are not smoothly concordant.
The knot Floer homology of the $K_i$ is:
\begin{equation}\label{eqn:K_iKnotFloerHomology}
\widehat{\mathrm{HFK}}(K_i,a;\mathbb{Q}) \cong \begin{cases} 
\mathbb{Q}^2_{(0)} & a = 1 \\
\mathbb{Q}^4_{(-1)} \oplus \mathbb{Q}_{(0)} & a = 0 \\
\mathbb{Q}^2_{(-2)} & a=-1
\end{cases}
\end{equation}
The detection results, conditions on $\Delta_J$, and $\widehat{\mathrm{HFK}}$ ribbon inclusion results imply, for either value of ${ i }$, that if $J \leq K_i$ and $J \neq K_i$ one must have 
$$\widehat{\mathrm{HFK}}(J,a;\mathbb{Q}) \cong \widehat{\mathrm{HFK}}(\overline{5_2},a;\mathbb{Q}) \cong \begin{cases} 
\mathbb{Q}^2_{(0)} & a = 1 \\
\mathbb{Q}^3_{(-1)} & a = 0 \\
\mathbb{Q}^2_{(-2)} & a=-1
\end{cases}
$$
$\overline{5_2}$ is detected by $\widehat{\mathrm{HFK}}$ \cite{BS25a}, so it will suffice to show that neither $K_i$ is smoothly concordant to $\overline{5_2}$. We have 
$$\tau(\overline{5_2})=-1 \neq 0=\tau(K_i). $$
Hence, both $Wh^-(T_{2,3},2)$ and $15n_{43522}$ are ribbon concordance minima.
\end{proof}
\end{proposition}

\subsection{Ribbon minimal knots which are not transfinitely nilpotent}
In the previous section, we showed ${ Wh^{\pm}(T_{2,3},2) }$ are ribbon minimal. We now prove Theorem \ref{thm:notTransfinitelyNilpotent}, showing they are not transfinitely nilpotent.
The argument uses some group-theoretic results which may be unfamiliar to low-dimensional topologists. We quickly review relevant terminology. 

A group ${ G }$ is \textit{torsion-free} if no non-identity element has finite order. 
We say ${ G }$ is \textit{locally indicable} if every non-trivial, finitely-generated subgroup ${H \leq G }$ possesses a surjective group homomorphism  ${ \phi : H \twoheadrightarrow \mathbb{Z} }$. 
Given a presentation of ${ G }$, a word ${ w }$ in this presentation is \textit{reduced} if it is impossible to reduce the length of the word by cancelling adjacent inverse letters; we say ${ w }$ is \textit{cyclically reduced} if every cyclic permutation of the letters of ${ w }$ is reduced. We say ${ G }$ is a \textit{1-relator group} if it admits a presentation containing a single relation. 
Our argument uses two theorems from the group theory literature.
\begin{theorem}[\cite{How81}, Theorem 4.3]\label{thm:Freiheitssatz}
Suppose ${ G = (A \ast B) / N }$ where ${ A  }$ and ${ B }$ are locally indicable groups and ${ N }$ is the normal closure in the free product ${ A \ast B }$ of a cyclically reduced word ${ r }$ of length at least 2. Then the canonical maps ${ A \rightarrow G }$ and ${ B \rightarrow G }$ are injective.  
\end{theorem}
\begin{theorem}[\cite{Bro80}, \cite{How2000}]\label{thm:localIndicability}
If ${ G }$ is a torsion-free, 1-relator group, then it is locally indicable.
\end{theorem}

We will also require a pair of basic lemmas whose proofs we supply for the reader's convenience. Recall the construction of the transfinite lower central series ${ \{\gamma_{\alpha}(G)\} }$ from the beginning of \S \ref{sec:surveyRibbonMinimality}.
\begin{lemma}\label{lem:rnilInheritance}
If $G$ is transfinitely nilpotent, then every subgroup of ${ G }$ is transfinitely nilpotent.
\begin{proof}
Given ${ H \leq G }$, it follows from transfinite induction that ${ \gamma_{\alpha}(H) \leq \gamma_{\alpha}(G) }$ for every ordinal ${ \alpha }$. 
In particular, for any finite ordinal $\alpha$ one has ${ \gamma_{\alpha}(H)\leq \gamma_{\alpha}(G) }$. Given the containment holds for some ordinal $\alpha$, one likewise has
$$ \gamma_{\alpha + 1}(H) = [\gamma_{\alpha}(H), H ] \leq [\gamma_{\alpha}(G), G ] = \gamma_{\alpha+1}(G)  $$ and so at a limiting ordinal one likewise has
$$ \gamma_{\lambda}(H) = \bigcap\limits_{\alpha < \lambda} \gamma_{\alpha}(H) \leq \bigcap\limits_{\alpha \leq \lambda} \gamma_{\alpha} (G) = \gamma_{\lambda} (G)      $$ 
If ${ G }$ is transfinitely nilpotent, then there is an ordinal ${ \alpha }$ such that ${ \gamma_{\alpha}(G)= \left\{1\right\} }$ and this forces ${ \gamma_{\alpha}(H) = \left\{1\right\} }$.
\end{proof}
\end{lemma}

\begin{lemma}\label{lem:stabilizationIfZabelianization}
If $G_{ab}= G / \gamma_{1}(G) \cong \mathbb{Z}$, then ${ \gamma_{1}(G) = \gamma_{i}(G)}$ for all ${ i \geq 1 }$.\footnote{Note this statement is stronger than ${ \gamma_{1}(G) \cong \gamma_{i}(G) }$ for all ${ i \geq 1 }$.}
\begin{proof}
The result follows inductively from showing ${ \gamma_{1}(G) = \gamma_{2}(G) }$. Consider the quotient ${ Q = G / \gamma_{2}(G) }$. Note that ${ \gamma_{2}(Q) = \left\{1\right\} }$ by construction and therefore ${ [Q,Q] }$ is central in ${ Q }$. Furthermore, using the commutator quotient formula and the fact that ${ \gamma_{2}(G) \leq \gamma_{1}(G) }$, we obtain
\begin{equation}\label{eqn:Qcommutator}
[Q,Q] = [G/ \gamma_{2}(G), G / \gamma_{2}(G)] = \Big([G,G] \gamma_{2}(G)\Big) / \gamma_{2}(G) \cong  \gamma_{1}(G) / \gamma_{2}(G)  
\end{equation}
We therefore have
$$ Q_{ab} = \frac{G / \gamma_{2}(G) }{ \gamma_{1}(G) / \gamma_{2} (G) } \cong G / \gamma_{1} (G) = G_{ab} \cong  \mathbb{Z}  $$ Let ${ t \in Q }$ be such that its image generates ${ Q_{ab} }$. Every element of ${ Q }$ can be written in the form ${ t^{n}z }$ where ${ z \in [Q,Q] }$. 
Given arbitrary elements ${ t^{m}z_{1}, t^{n}z_{2} \in Q }$, the fact that ${ [Q,Q] }$ is central in ${ Q }$ implies
$$[t^{m} z_{1} , t^{n} z_{2} ] = [t^{m},t^{n}  ] =    1         $$ 
which means that ${ Q }$ is abelian and therefore ${ [Q,Q] = \left\{1\right\} }$. Equation \ref{eqn:Qcommutator} therefore implies ${ \gamma_{1}(G) / \gamma_{2}(G) = \{1\} }$ and therefore ${ \gamma_{1}(G) = \gamma_{2}(G) }$.
\end{proof}
\end{lemma}

\begin{theoremrep}{thm:notTransfinitelyNilpotent}
${ K_{+} = Wh^{+}(T_{2,3},2) }$ is neither ${ \mathbb{Q} }$-anisotropic nor transfinitely nilpotent. ${ K_{-} = Wh^{-}(T_{2,3},2) }$ is ${ \mathbb{Q} }$-anisotropic but not transfinitely nilpotent.
\begin{proof}
Recall $\Delta_{K_{+}} = (2t-1)(t-2)$. Set $\widetilde{X}$ to be the infinite cyclic cover of $X := S^3 \setminus \nu K_{+}$. The Alexander module of $K_{+}$ is semisimple, admitting a decomposition into two 1-dimensional $t$-invariant subspaces
$$H^1(\widetilde{X}; \mathbb{Q})\cong \mathrm{Hom}\left(\frac{\mathbb{Q}[t^{\pm 1}] }{\langle (t-2 )(2t-1 )  \rangle}, \mathbb{Q}\right ) \cong \mathbb{Q}_{(2)}  \oplus \mathbb{{Q}}_{(1 /2)} $$
where subscripts indicate the eigenvalue of the ${ t }$-action on the relevant subspace. Skew-symmetric bilinear forms vanish on 1-dimensional subspaces, hence $K_{+}$ is not $\mathbb{Q}$-anisotropic.

We now show that ${ K_{+} }$ is not transfinitely nilpotent.
Set ${ G = \pi_{1}(S^{3} \setminus \nu K_{+} ) }$, using \texttt{SnapPy} \cite{SnapPy} we obtain the presentation 
$$
G  \cong \left\langle a,b,c \;\middle|\;  
\begin{array}{rcl}  
r_{1} &=& a^{2}c^{2}bc^{-1} bca^{-1} c^{-1} b^{-1} cb^{-1} c^{-2} b^{-1} cb^{-1} c^{-2} a^{-1} c^{2} bc^{-1} bc \\  
r_{2} &=& b^{2}cb^{-1}c  
\end{array}  \;
\right\rangle
$$
The above yields a presentation ${ G_{ab} \cong \langle a, b, c \; | \; r_{1}' = c, \;r_{2}'=bc^{2} \rangle \cong \langle a \rangle \cong   \mathbb{Z}}$.
Consider the group ${ H = \langle b, c \; | \; r_{2} \rangle }$ and observe that ${ G \cong H \ast \langle a \rangle / \langle\! \langle r_{1} \rangle\! \rangle  }$ where ${ \langle\! \langle r_{1} \rangle\! \rangle }$ denotes the normal closure of the word ${ r_{1} }$ inside the free product ${ H \ast \langle a \rangle }$. The relator ${ r_{2} }$ is not a proper power, so ${ H }$ is a torsion-free, 1-relator group and therefore by Brodski{\u\i}'s theorem (Theorem \ref{thm:localIndicability}) ${ H }$ is locally indicable. ${ \langle a \rangle \cong \mathbb{Z} }$ is likewise locally indicable and ${ r_{1} }$ is cyclically reduced, so Howie's theorem (Theorem \ref{thm:Freiheitssatz}) tells us that the canonical group homomorphism ${ H \rightarrow G }$ is injective--- we abuse notation and let ${ H }$ denote the isomorphic subgroup of ${ G }$.
One has ${ \mathbb{Z} \cong G_{ab} = G / \gamma_{1}(G) }$ is generated by the image of ${ a }$ under the quotient map, so ${ b, c \in \gamma_{1}(G) }$ and we conclude ${ H \leq \gamma_{1}(G) }$. 
Our goal is to show that ${ \gamma_{1}(G) }$ is not transfinitely nilpotent--- we do so by showing ${ H }$ is not transfinitely nilpotent and then applying Lemma \ref{lem:rnilInheritance}.
Note ${ H / \gamma_{1}(H) = H / [H,H] \cong \langle b, c \; | \; bc^{2}\rangle \cong \mathbb{Z}  }$ and so Lemma \ref{lem:stabilizationIfZabelianization} tells us that ${ \gamma_{i}(H) = \gamma_{1}(H) }$ for all ${ i \geq 1 }$. 
Note ${ \gamma_{1}(H) = \{1\} }$ if and only if ${ H }$ is abelian, however the group homomorphism
$${ \phi : H \rightarrow S_{3} \qquad   \phi: \begin{cases} b \mapsto  (123) \\  c \mapsto (23) \end{cases} \qquad \text{ satisfies } \qquad \phi(bc)\neq \phi(cb). }$$
Hence, ${ \{1\}\neq \gamma_{1}(H) = \gamma_{\alpha}(H) }$ for all ordinals ${ \alpha > 1 }$ and it follows that ${ H }$ is not transfinitely nilpotent and therefore neither is ${ [G,G] }$. 
\newline
We now turn to ${ K_{-} = Wh^{-}(T_{2,3},2) }$. Since ${ \Delta_{K_{-}} }$ is irreducible, it follows immediately that ${ K }$ is ${ \mathbb{Q} }$-anisotropic. 
Let ${ G = \pi_{1}(S^{3} \setminus \nu K_{-}) }$; we can show that ${ [G,G] }$ is not transfinitely nilpotent using a very similar strategy.  Using \texttt{SnapPy} \cite{SnapPy} we obtain the presentation 
$$
G  \cong \left\langle a,b,c \;\middle|\;  
\begin{array}{rcl}  
r_{1} &=& a c^{-1} b c^{-1} b^{-3} a^{-1} b^3 c b^{-1} c a^{-1} b^{-1} c^{-1} b c^{-1} b^{-2} a b^4 c b^{-1} c \\  
r_{2} &=& b c^{-1} b c^2  
\end{array}  \;
\right\rangle
$$

We can repeat the same strategy as we did for ${ Wh^{+}(T_{2,3},2) }$. In particular we observe that in ${ G_{ab} }$ the relations descend to ${ r_{1}' = b }$ and ${ r_{2}' = b^{2}c }$, so ${ G_{ab} }$ is generated by the image of ${ a }$. Setting ${ H = \langle b,c \; | \; r_{2} \rangle  }$, we can show that ${ H \rightarrow  G \cong H \ast \langle a \rangle / \langle \! \langle r_{1} \rangle \! \rangle }$ is injective using Howie's theorem. Since ${ H_{ab} \cong \mathbb{Z} }$, we conclude that ${ \gamma_{1}(H) = \gamma_{i}(H) }$ for all ${ i \geq 1 }$ and we can show ${ \gamma_{1}(H) \neq \left\{1\right\} }$ by means of the homomorphism
$${ \phi : H \rightarrow S_{3} \qquad   \phi: \begin{cases} b \mapsto  (12) \\  c \mapsto (123) \end{cases}\;. }$$
\end{proof}
\end{theoremrep}

\subsection{Link Floer homology \& ribbon minima}\label{sec:linkMinima}

Suppose we wish to show $L$ is ribbon minimal. As in the knot case, one considers a counterfactually distinct $L'$ with $L' \leq L$ and tries to impose constraints on $\Delta_{L'}^{tor}$ which force $L'$ to have the same link Floer homology. For our analysis of 2-component minimal links, a result of Davis plays a similar r\^{o}le to Freedman's theorem in the knot case.  

\begin{theorem}[\cite{Dav06}]\label{thm:DavisThm}
A 2-component link with Alexander polynomial $\pm 1$ is topologically concordant to a Hopf link.
\end{theorem}

 Given fixed ${ n \in \mathbb{Z} }$, we let $H_{n}^{-}$ denote ${ T(2,2n) }$ with its components oriented oppositely. 
\begin{theoremrep}{thm:HopfLinksMinimal}
The links $H_n^-$ are ribbon minimal for all $n$.
\begin{proof}
In Theorem 3.1 of \cite{BM24}, it is shown that $\widehat{\mathrm{HFK}}$ detects $H_n^-$ and, since knot Floer homology can be recovered from link Floer homology by applying a shift to the Maslov grading, 
 it follows that $\widehat{\mathrm{HFL}}$ detects $H_n^-$. The link Floer homology of $H_n^-$ was computed in \S 12.1 of \cite{OS08a}, for the reader's convenience we present it in Equation \ref{eqn:HFL_H_minus_n} below.
\begin{equation}\label{eqn:HFL_H_minus_n}
\widehat{\mathrm{HFL}}(H_{n}^-,a_1,a_2)\cong 
\begin{cases}
& 
\includegraphics[width = 0.45\linewidth]{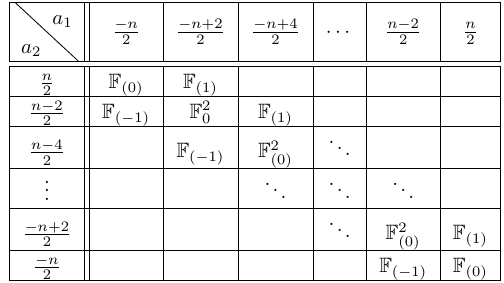}
\end{cases}
\end{equation}
As before, the subscripts indicate the Maslov grading of a summand and $(a_1,a_2) \in \mathbb{H}$ is an Alexander multigrading. Applying Equation \ref{eqn:LinkFloerEulerChar} to the above, one obtains
$$ \Delta_{H_{n}^{-}  }(t_1,t_2) = t_{1}^{\frac{1-n}{2}} t_{2}^{\frac{1-n}{2}} \sum\limits_{i=0}^{n-1} (t_{1} t_{2}   )^{i}  = t_{1}^{\frac{1-n}{2}} t_{2}^{\frac{1-n}{2}}\frac{(t_{1}t_{2}  )^{n} -1 }{t_{1} t_{2} -1}    $$
and, since the above is non-zero, we have $\Delta_{H_n^-} = \Delta_{H_n^-}^{tor}$.

Suppose $L = L_1 \sqcup L_2$ is a 2-component link such that ${ L \leq H_{n}^{-} }$.
Note there is no subspace of $\widehat{\mathrm{HFL}}(H_n^-)$ with graded Euler characteristic zero, hence the grading-preserving injection $\widehat{\mathrm{HFL}}(L) \hookrightarrow \widehat{\mathrm{HFL}}(H_n^-)$ induced by a ribbon concordance from $L$ to $H_n^-$ forces $0\neq \Delta_L = \Delta_L^{tor}$.\footnote{Recall from Definition \ref{defn:AlxPolys} that when the Alexander module is not torsion, then one defines $\Delta_L = 0$. Since $\Delta_L^{tor}$ is the generator of the order ideal of the torsion part of the Alexander module, it follows that $\Delta_L \neq 0$ implies $\Delta_L = \Delta_L^{tor}$.} 
Milnor \cite{Mil54} proved that the linking number is link homotopy invariant, and therefore a topological concordance invariant \textit{a fortiori}. Hence, the condition $L \leq H_n^-$ implies  $\ell k (L_1,L_2)=-n$. 

The Torres condition (Theorem \ref{thm:TorresCondition}) for the Alexander polynomial states that 
$$\Delta_{L}(t_1,1)= \frac{t_1^{n}-1}{t_1-1} \cdot \Delta_{L_1}(t_1) \qquad \text{ and } \qquad \Delta_{L}(1,t_2)= \frac{t_2^{n}-1}{t_2-1} \cdot \Delta_{L_2}(t_2)$$
Since $\Delta_L = \Delta_L^{tor}$ divides $\Delta_{H_n^-} = \Delta_{H_n^-}^{tor}$ the above implies $\Delta_L \in \{1,\Delta_{H_n^-}\}$. Since the linking number is invariant under topological concordance, Theorem \ref{thm:DavisThm} tells us $\Delta_L=1$ if and only if $n=1$, hence for all cases we have $\Delta_L = \Delta_{H_n^-}$.

The only $\mathbb{F}$-vector space which could include into ${ \widehat{\mathrm{HFL}}(H_{n}^{-}) }$ and have the same graded Euler characteristic is ${ \widehat{\mathrm{HFL}}(H_{n}^{-}) }$ itself, so we conclude $\widehat{\mathrm{HFL}}(L) \cong \widehat{\mathrm{HFL}}(H_{n}^-)$ and $L = H_n^-$ by \cite{BM24}.
\end{proof}
\end{theoremrep}

Let ${ W_{n} }$ denote the link obtained by taking the $n^{th}$ twist knot ${ T_{n} }$, situated in the solid torus, union a meridian ${ \mu }$ of the solid torus. 
We require some lemmas.
\begin{lemma}\label{lem:WnNotConcordant}
For ${ m , n \in \mathbb{Z} \setminus \{4,-5\} }$, the links ${ W_{m} }$ and ${ W_{n} }$ are not strongly smooth concordant to one another.
\begin{proof}
The basic idea in the proof is similar to an argument found in \cite{Tag23}.
Suppose $W_m$ and $W_n$ are strongly smooth concordant for $m\neq n$, this forces the existence of a concordance between the twist knots $T_n$ and 
$T_m$.\footnote{Either the link concordance contains a cylinder cobounded by $T_m$ and $T_n$ or contains two cylinders that show both $T_m$ and $T_n$ are smoothly concordant to the unknot.}
Recall that the branched double cover of $T_n$ is the lens space  $L(4n+1,2n)$ and if $T_n$ and $T_m$ are smoothly concordant then
$\Sigma_2(T_n \# -T_m)\cong L(4n+1,2n) \# -L(4m+1,2m)$ bounds a rational homology ball. 
The main theorem of \cite{Lis07} completely characterizes when this can happen and implies that the only twist knots which are smoothly concordant to one another are
$T_0$ (the unknot) and $T_4$ (the Stevedore knot), and $T_{-5}$ (mirror of Stevedore).
\end{proof}
\end{lemma}

\begin{lemma}\label{lem:ATtwistMinimality}
With the exception of ${ T_{4} }$ (the Stevedore knot, $6_1$) and $T_{-5}$ (its mirror), all twist knots are ribbon concordance minimal.
\begin{proof}
In \cite{AT24}, Corollary 1.4 it is shown that if a 2-bridge knot is smoothly concordant to a torus knot, then it is ribbon concordant to a torus knot. 
Moreover, Theorem 1.3 of \cite{AT24} states that a 2-bridge knot which is not a torus knot is smoothly concordant to a torus knot if and only if it's of the form ${ K(m^{2}n, mnk+\epsilon) }$ where ${ m,k }$ are coprime integers satisfying  ${ m > k > 0 }$ and ${ \epsilon \in \left\{\pm 1\right\} }$. Twist knots have presentation ${T_{r} =  K(2r + 1 , r) }$, so if ${ K_{r} }$ is concordant to a torus knot, then we would have 
$$ 2r + 1 = m^{2} n \;\; \& \;\; r = mnk + \epsilon \qquad \Rightarrow \qquad mn(m -2k)= 2\epsilon +1 $$ Since ${ 2\epsilon + 1 \in \left\{-1,3\right\} }$ and ${ m > k > 0 }$, we conclude that ${ mn(m-2k) = 3 }$ and therefore ${ m = 3, n = 1, k=1 , }$ and ${ \epsilon = 1 }$ so the only twist knot satisfying the hypotheses of Abe and Tagami's theorem is ${ T_4 = K(9,4) = 6_{1}}$ and its mirror. 
\end{proof}
\end{lemma}

\begin{theoremrep}{thm:WhiteheadLinks}
The ${ n }$-twisted Whitehead links ${ W_{n} }$ are ribbon minimal for ${ n \neq 4,-5 }$.
\begin{proof}
In \cite{BD24}, it is shown that ${ \widehat{\mathrm{HFL}} }$ detects ${ W_{n} }$ when ${ n>1 }$ and detects membership in ${ \left\{W_{0}, W_{1}\right\} }$. 
We first note that ${ \overline{W_{n}} = W_{-n-1} }$, hence it suffices to prove ribbon minimality for ${ n \geq 0 }$.
We furthermore note that for any $n$, one has
$$\Delta_{W_n}^{tor} = \Delta_{W_n} = (t_{1}^{1 / 2} - t_{1}^{-1 / 2})(t_{2}^{1 / 2} - t_{2}^{-1 / 2})$$
The argument for minimality splits into two cases depending on the value of ${ n }$.
\vspace{5pt}

\sectionrule
\vspace{5pt}
\noindent
\textbf{Case i:} ${ n \in \left\{0,1\right\} }$ 
\vspace{5pt}

\sectionrule 
\vspace{5pt}
\noindent
We have
\begin{equation}\label{eqn:W0or1}
\widehat{\mathrm{HFL}}(W_{0})\cong \widehat{\mathrm{HFL}}(W_1) \cong
\begin{cases}
& 
\includegraphics[width = 0.25\linewidth]{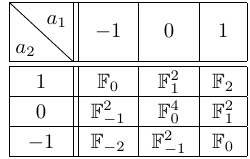}
\end{cases}
\end{equation}
We remark that in the above the Alexander grading $a_1$ corresponds to the twist knot component and $a_2$ corresponds to the unknot component.

In the $n \in \{0,1\}$ case, if $L \leq W_n$ then the ribbon-inclusion $\widehat{\mathrm{HFL}}(L) \hookrightarrow \widehat{\mathrm{HFL}}(W_n)$ must have non-zero Euler characteristic, so we have $0 \neq \Delta_L = \Delta_L^{tor}$.  Gilmer's Theorem (Theorem \ref{thm:GilmersTheorem}) and the Torres condition implies ${ \Delta_{L} \in \left\{1, \Delta_{W_{n}} \right\} }$.
None of the ${ W_{n} }$ are topologically concordant to the Hopf link and therefore, by \cite{Dav06} (Theorem \ref{thm:DavisThm}), are not concordant to any 2-component link with multivariable Alexander polynomial 1. 
Hence ${ \Delta_{L} = \Delta_{W_{n}} }$ and, in particular, the injection ${ \widehat{\mathrm{HFL}}(L) \hookrightarrow \widehat{\mathrm{HFL}}(W_{n}) }$ must preserve the Euler characteristic.
For $n \in \{0,1\}$, the only subspace of $\widehat{\mathrm{HFL}}(W_n)$ having the same Euler characteristic is itself. Theorem 6.2 of \cite{BD24} implies that $L \in \{W_0,W_1\}$. Lemma \ref{lem:WnNotConcordant} tells us that $W_0$ and $W_1$ are not smoothly concordant, so we conclude $L = W_n$.

\vspace{5pt}

\sectionrule
\vspace{5pt}
\noindent
\textbf{Case ii:} ${ n > 1 }$ and ${ n \neq 4 }$ 
\vspace{5pt}

\sectionrule 
\vspace{5pt}
\noindent

When ${ n>1 }$, Binns \& Dey compute $\widehat{\mathrm{HFL}}(W_n)$ to be as in Figure \ref{fig:BinnsDeyWhiteheadLinkHFL} below.

\begin{figure}[ht]
\centering

\begin{subfigure}[t]{0.48\textwidth}
\centering
\includegraphics[width=\textwidth]{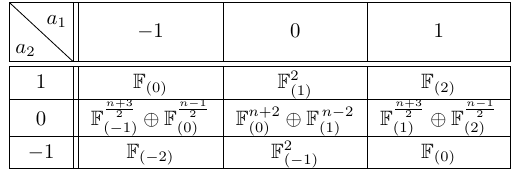}
\caption{$n>1$ odd}\label{fig:W_nOdd}
\end{subfigure}
\hfill
\begin{subfigure}[t]{0.48\textwidth}
\centering
\includegraphics[width=\textwidth]{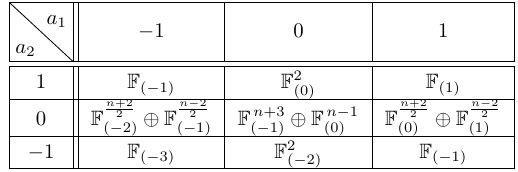}
\caption{$n>1$ even}\label{fig:W_nEven}
\end{subfigure}

\caption{$\widehat{\mathrm{HFL}}(W_n)$ for $n>1$ as computed in \cite{BD24}, Proposition 6.3. We remind the reader that ${ a_{1} }$ is the Alexander grading corresponding to the twist knot component while ${ a_{2} }$ corresponds to the unknot component.}
\label{fig:BinnsDeyWhiteheadLinkHFL}
\end{figure}
We focus on the case where ${ n >1 }$ is odd as the even case follows identically.
The strong ribbon concordance tells us there is a subspace of  ${ \widehat{\mathrm{HFL}}(W_{n},\vec{a}) }$ which is the image of ${ \widehat{\mathrm{HFL}}(L,\vec{a}) }$ under inclusion. 
This subspace is the link Floer homology of a link and so it must possess a symmetry
\begin{equation}\label{eqn:symmetryInGrading} 
\widehat{\mathrm{HFL}}_{m}(L,\vec{a}) \cong \widehat{\mathrm{HFL}}_{m-2\sum_{i}a_{i}  } (L,-\vec{a})  
\end{equation}

It follows that ${ \widehat{\mathrm{HFL}}(L,\vec{a}) }$ will have the form

\begin{equation} \label{eqn:L_concordantWnodd}
\widehat{\mathrm{HFL}}(L, \vec{a})  = \begin{cases} 
\includegraphics[height=2.5cm]{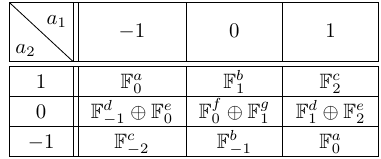}  
\end{cases}
\;\;
{\small  
\begin{aligned}  
&a,c \in \{0,1\}, & b &\in \{0,1,2\}  \\  
0 &\le d \le \tfrac{n+3}{2}, & 0 &\le e \le \tfrac{n-1}{2} \\  
0 &\le f \le n+2, & 0 &\le g \le n-2  
\end{aligned}  
}  
\end{equation} 

Recall that ${ W_{n} =   T_{n} \sqcup \mu}$. By Lemma \ref{lem:ATtwistMinimality}, both components of ${ W_{n} }$ are ribbon concordance minimal, it follows ${ L \leq W_{n} }$ contains a component ${ L' }$ which is isotopic to ${ T_{n} }$ and corresponding to the ${ a_{1} }$-Alexander grading in Equation (\ref{eqn:L_concordantWnodd}).
Using this, we obtain further constraints on ${ \widehat{\mathrm{HFL}}(L,\vec{a}) }$ by means of the ``component removal"
spectral sequence described in Proposition 7.1 of \cite{OS08a}. 
The $E_2$-page of this spectral sequence can be obtained from ${ \widehat{\mathrm{HFL}}(L) }$ by collapsing the ${ a_{2} }$-grading.
From Equation \ref{eqn:L_concordantWnodd}, we obtain
$$ E_2 \cong \begin{cases}  \mathbb{F}_{(0)}^{a} \oplus \mathbb{F}_{(1)}^{d}  \oplus \mathbb{F}_{(2)}^{c+e}   & a_{1}  = 1 \\ \mathbb{F}_{(-1)}^{b} \oplus \mathbb{F}_{(0)}^{f} \oplus  \mathbb{F}_{(1)}^{b+g}    & a_{1}  = 0 \\ \mathbb{F}_{(-2)}^{c}   \oplus \mathbb{F}_{(-1)}^{d} \oplus \mathbb{F}_{(0)}^{a+e}    & a_{1}  = -1    \end{cases} $$

Using the computation of ${ \widehat{\mathrm{HFK}}(T_{n}) }$ for ${ n }$ odd, the $E_{\infty}$-page will be

\begin{equation}\label{eqn:E_inftyPage}
E_{\infty} \cong \widehat{\mathrm{HFK}}(T_{n} ) \otimes (\mathbb{F}_{(-1)} \oplus \mathbb{F}_{(0)}  )  \cong  \begin{cases} \mathbb{F}_{(1)}^{\frac{n+1}{2}} \oplus \mathbb{F}_{(2)}^{\frac{n+1}{2}}   & a_{1}  = 1  \\
 \mathbb{F}_{(0)}^{n} \oplus \mathbb{F}_{(1)}^{n}   & a_{1}  = 0  \\
 \mathbb{F}_{(-1)}^{\frac{n+1}{2}} \oplus  \mathbb{F}_{(0)}^{\frac{n+1}{2}}    & a_{1}  = -1    \end{cases} 
 \end{equation}
Examining ranks in the relevant gradings of $E_2$ and $E_{\infty}$, it is clear that in order for $E_{\infty}$ have the form seen in (\ref{eqn:E_inftyPage}), one must have ${ \widehat{\mathrm{HFL}}(L,\vec{a}) \cong \widehat{\mathrm{HFL}}(W_{n},\vec{a}) }$ and therefore ${ L = W_{n} }$ by \cite{BD24}, Theorem 6.1. 
\end{proof}
\end{theoremrep}

\begin{theoremrep}{thm:componentsOfLConcordantToW4}
Suppose ${ L \leq W_{4} }$. Letting ${ L' }$ be the component of ${ L }$ which corresponds to ${ T_{4} \subset W_{4} }$, ${ L' }$ is either the unknot or ${ T_{4} }$. 
Moreover, if ${ L' }$ is ${ T_{4} }$, then ${ L = W_{4} }$.
\begin{proof}
Specializing Table \ref{fig:W_nEven} to ${ n = 4 }$, the assumption ${ L \leq W_{4} }$ yields an inclusion of subspaces

\begin{equation}\label{eqn:W_4subspace}
\widehat{\mathrm{HFL}}(L, \vec{a}) \hookrightarrow \widehat{\mathrm{HFL}}(W_{4} ) \cong \begin{cases} {\includegraphics[height=2.5cm]{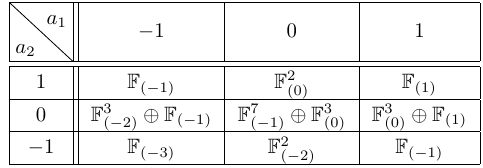}} \end{cases}
\end{equation} 
Note there are subspaces of ${ \widehat{\mathrm{HFL}}(W_{4}) }$ of graded Euler characteristic zero. In the proof of Theorem \ref{thm:WhiteheadLinks} we used minimality of the twist knot component to show such subspaces will never be the injective image of $\widehat{\mathrm{HFL}}(L)$; here we must work slightly harder to show $\Delta_L\neq 0$ by first showing this would force $L$ to be split. The Sato-Levine invariant $\beta$, see \cite{Cha06} for the definition, vanishes on links which are strongly concordant to split links. One can easily show that $\beta(W_4)=1$ and therefore is not strongly concordant to any split link. 

Supposing ${ \Delta_{L} = 0 }$, an examination of (\ref{eqn:W_4subspace}) forces ${ \widehat{\mathrm{HFL}}(L) }$ to be supported exclusively in Alexander gradings $(a_1,0)$ where ${ a_{1} \in \left\{-1,0,1\right\} }$.
Letting ${ Y = S^{3} \setminus \nu L }$, the meridians of ${ L }$ give a basis ${ \{ [\mu_{1}], [\mu_{2}] \} }$ for ${ H_{1}(Y;\mathbb{Z}) }$ and so there is a dual basis ${ \{ h^{(1)} , h^{(2)} \} }$ for ${ H_{2}(Y,\partial Y ; \mathbb{Z}) \cong H^{1}(Y; \mathbb{Z}) }$. 

In \cite{OS08b}, a function on ${ H_{2}(Y,\partial Y) }$ is described using the Alexander gradings of ${ \widehat{\mathrm{HFL}}(L) }$:
$$ y_{L} (h)  = \max \left\{|\langle PD(h) , s \rangle|  \;  \middle|\; s \in \mathbb{H}_L \cong H_{1}(S^{3} \setminus \nu L), \; \widehat{\mathrm{HFL}}(L,s ) \neq 0  \right\}  $$
We write $\chi(h)$ for the maximal Euler characteristic across all embedded surfaces $\Sigma$ such that \newline ${[\Sigma] = h \in H_2(Y, \partial Y)}$.
In \cite{Ni09}, a formula is given relating ${ y_{L}(h) }$ to $\chi(h)$ for any class $h \in H_2(Y,\partial Y)$:
$$ -\chi(h) + \sum\limits_{i=1}^{2} | \langle PD(h) , [\mu_{i} ] \rangle | = 2y_{L} (h)  $$
Note that ${ y_{L}(h^{(2)}) = 0  }$ and by the formula above, we conclude ${ \chi(h^{(2)}) = 1 }$. As pointed out in Remark 1.3 of \cite{Ni09}, this shows that one of the components of ${ L }$ contributes a boundary torus of ${ S^{3} \setminus \nu L}$ which is compressible. Performing disk surgery on this boundary torus, one obtains a splitting sphere separating the two components of ${ L }$, so  ${ \Delta_{L} =0 }$ forces ${ W_{4} }$ to be concordant to a split link but, as mentioned earlier, this is obstructed by the Sato-Levine invariant.
Combined with Gilmer's theorem we conclude $\Delta_{L} = \Delta_{L}^{tor} \in \{1, \Delta_{W_4}\}$. The linking number obstructs $W_4$ from being topologically concordant to the Hopf link, so Davis' theorem implies ${ \Delta_{L} = \Delta_{W_{4}} }$. 

Since ${ \widehat{\mathrm{HFL}}(L) }$ and ${ \widehat{\mathrm{HFL}} (W_{4})}$ have the same graded Euler characteristic, one is forced to have
$$ \forall \vec{a} \notin \left\{(\pm 1, 0), (0,0)\right\} : \qquad  \widehat{\mathrm{HFL}}(L, \vec{a}) = \widehat{\mathrm{HFL}}(W_{4} , \vec{a}) $$
Accounting for the symmetries required of link Floer homologies and the fact that ${ \Delta_{L} = \Delta_{W_{4}} }$, we conclude that there are \textit{a priori} independent constants ${ a \in \left\{2, 3\right\} }$ and ${ b \in \left\{4,5,6,7\right\} }$ such that 

\begin{equation}\label{eqn:W_4subspaceReduced}
\widehat{\mathrm{HFL}}(L, \vec{a})  = \begin{cases} \includegraphics[height=3cm]{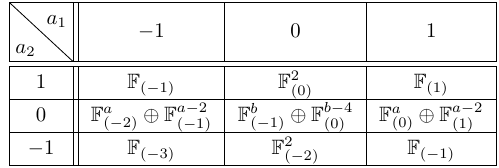} \end{cases}
\end{equation}  

Note that the non-unit factors of  ${ \Delta_{T_{4}} = -(2t-1)(t-2) }$ cannot be the Alexander polynomial of a knot, so Gilmer's theorem tells us that if ${ K \leq T_{4} }$, then ${ \Delta_{K} \in \left\{1 , \Delta_{T_{n}}\right\} }$. The knot Floer homology of ${ T_{4} }$ is 
$$ \widehat{\mathrm{HFK}}(T_{4} , a) \cong \begin{cases} \mathbb{F}_{(1)}^{2} & a = 1  \\
\mathbb{F}_{(0)}^{5}  & a = 0 \\
\mathbb{F}_{(-1)}^{2}  &  a = -1  \end{cases} $$
The only subspace of ${ \widehat{\mathrm{HFK}}(T_{4}) }$ of Euler characteristic one is the knot Floer complex of the unknot, likewise, the only subspace of Euler characteristic ${ \Delta_{T_{4}} }$ is the entire space. 
The knot Floer homology detects the unknot and Baldwin and Sivek \cite{BS25a} show that it likewise detects ${ T_{4} = 6_{1} = P(-3,3,1) }$.  It follows then that any ${ K }$ satisfying ${ U \leq K \leq T_{4} }$ is equal to either ${ U }$ or ${ T_{4} }$. This tells us that the requisite strong ribbon concordance from ${ L }$ to ${ W_{4} }$ connects the component ${ T_{4} \subset W_{4} }$ to a component ${ K \subset L }$ which, as a knot, is isotopic to either the unknot or to another copy of ${ T_{4} }$. Supposing we have ${ K = T_{4} }$, one may use the same argument as that used in the proof of Theorem \ref{thm:WhiteheadLinks} to show that $L$ has the same link Floer homology as $W_4$ and therefore by \cite{BD24} one has $L = W_4$.
\end{proof}
\end{theoremrep}

\section{Questions and conjectures}

By \cite{Ba16} $L$-space knots are ribbon concordance minimal, and alternative proof can also be obtained by combining \cite{BS22} with \cite{BG24}. 
${ L }$-space links behave quite differently from their knot counterparts; they are not necessarily fibered nor strongly quasipositive (see \cite{CL23} for details). 

\begin{question}
Are all ${ L }$-space links ribbon concordance minimal?
\end{question}

We proved earlier that two twisted Whitehead doubles of a ribbon minimal knot are themselves ribbon minimal. It is tempting to conjecture that Whitehead operators preserve minimality. More generally, we ask

\begin{question}
Are there any satellite operators that preserve ribbon minimality? 
\end{question}

Freedman's Theorem states that $\Delta_K=1$ implies $K$ is topologically slice; 
one can also show $\Delta_K=1$ implies $K$ is neither $\mathbb{Q}$-anisotropic nor transfinitely nilpotent,\footnote{Letting ${ G = \pi_{1}(S^{3 } \setminus \nu K) }$ and ${ \Delta_{K} = 1 }$ and ${ K\neq U }$, then this means commutator subgroup of $G$ is perfect and therefore its lower central series is constant. See \cite{Joh23} Proposition 1.1.} so Gordon's criterion cannot address these knots. 
On the other hand, it is reasonable to believe there are $K$ as above which are ribbon minimal. The Conway knot $C = 11n_{34}$ is a particularly conspicuous candidate; it has Alexander polynomial 1 and is the only knot with $\leq 12$ crossings which is topologically slice but not smoothly slice. \cite{Pic20} \cite{DG25} 
\begin{conjecture}\label{conj:ConwayMinimality}
The Conway knot is ribbon minimal.
\end{conjecture}
Computer experimentation reported in \cite{DG25} provides evidence for ${ C }$ being ribbon minimal.
Moreover, if there were some knot $J\leq C$ and $J\neq C$, then the crossing number of $J$ would be bigger than that of $C$. This would furnish a counter-example to a folk conjecture: that $J\leq K$ then the crossing number of $J$ is less than or equal to that of $K$.

By exhausting all currently existing $\widehat{\mathrm{HFK}}$ knot detection results, we obtained three ribbon minimal knots. It is reasonable to wonder if there are underlying geometric phenomena that could lead to an alternative proofs of their minimality.
\begin{question}
Can one prove that ${ Wh^{\pm}(T_{2,3},2) }$ and ${ 15n_{43522} }$ are ribbon minimal without using Heegaard Floer detection? In particular, is there some property, ideally shared with an infinite family, that forces their minimality?
\end{question}

We believe but do not have a proof of either of the following conjectures.
\begin{conjecture}\label{conj:minimalKnotWithMeridian}
Let ${ K }$ be a knot $S^3$ and let ${ \mu_{K} }$ be a meridian of ${ K }$. Then ${L_K = K \sqcup \mu_{K} }$ is a ribbon concordance minimal link if and only if $K$ is ribbon minimal.
\end{conjecture}

If $L \leq L_K$ and $L \neq L_K$, then the hypotheses imply $L$ contains the same knots as components, but the way in which they are linked has been modified and, in particular, the unknot component of $L$ will not be a meridian of the other component. 
We expect a general result resolving 
Conjecture \ref{conj:minimalKnotWithMeridian} to find application in studying minimality of satellite knots.

\begin{conjecture}\label{conj:nonminimalComponent}
There exists a knot $K$ which is not ribbon minimal but is contained in a ribbon minimal link as a component.
\end{conjecture}

In Remark \ref{rem:W4minimality} we advocated for the Stevedore knot as a potential example; its reasonableness as a candidate is proportional to the reasonableness that it bounds finitely many ribbon disks up to isotopy rel. boundary and modulo local knotting. On the other hand, it is shown in Theorem 1.2 of \cite{MM25} that the ``generalized square knots" bound infinitely many ribbon disks up to isotopy rel. boundary and modulo local knotting. We ask the following.

\begin{question}
Are there any knots ${ K }$ such that any link ${ L }$ containing ${ K }$ is not ribbon concordance minimal?
\end{question}

\bibliographystyle{alpha}
\bibliography{bibliography}

@article{Gor81,
  title={Ribbon concordance of knots in the 3-sphere},
  author={Gordon, C McA},
  journal={Mathematische Annalen},
  volume={257},
  number={2},
  pages={157--170},
  year={1981},
  publisher={Springer}
}

@misc{SnapPy,
  author = {Culler, Marc and Dunfield, Nathan M. and Goerner, Matthias and Weeks, Jeffrey R.},
  title  = {SnapPy, a computer program for studying the geometry and topology of 3-manifolds},
  howpublished = {\url{http://snappy.computop.org}},
  note = {Accessed: 2026-06-02}
}

@inproceedings{Dav06,
  title={A two component link with {A}lexander polynomial one is concordant to the Hopf link},
  author={Davis, James F},
  booktitle={Mathematical Proceedings of the Cambridge Philosophical Society},
  volume={140},
  number={2},
  pages={265--268},
  year={2006},
  organization={Cambridge University Press}
}

@article{BS25b,
  title={Ribbon concordance and fibered predecessors},
  author={Baldwin, John A and Sivek, Steven},
  journal={arXiv preprint arXiv:2510.02214},
  year={2025}
}

@article{DLVW22,
  title={Ribbon homology cobordisms},
  author={Daemi, Aliakbar and Lidman, Tye and Vela-Vick, David Shea and Wong, C-M Michael},
  journal={Advances in Mathematics},
  volume={408},
  pages={108580},
  year={2022},
  publisher={Elsevier}
}

@article{DG25,
  title={Ribbon concordances and slice obstructions: experiments and examples},
  author={Dunfield, Nathan M and Gong, Sherry},
  journal={arXiv preprint arXiv:2512.21825},
  year={2025}
}

@article{BHS26,
  title={Ribbon concordance and fibered predecessors, II: the general case},
  author={Baldwin, John A and Hanselman, Jonathan and Sivek, Steven},
  journal={arXiv preprint arXiv:2602.21109},
  year={2026}
}

@article{Joh23,
  title={Residual torsion-free nilpotence, biorderability and pretzel knots},
  author={Johnson, Jonathan},
  journal={Algebraic \& Geometric Topology},
  volume={23},
  number={4},
  pages={1787--1830},
  year={2023},
  publisher={Mathematical Sciences Publishers}
}

@article{How2000,
  title={A short proof of a theorem of {B}rodski{\u\i}},
  author={Howie, James},
  journal={Publicacions Matem{\`a}tiques},
  pages={641--647},
  year={2000},
  publisher={JSTOR}
}

@article{Bro80,
  author  = {Brodski{\u{\i}}, S. D.},
  title   = {Equations over groups and groups with a single defining relation},
  journal = {Russian Mathematical Surveys},
  volume  = {35},
  number  = {4},
  year    = {1980},
  pages   = {165},
  note    = {English translation of Uspekhi Mat. Nauk 35, no. 4 (1980), 183}
}

@article{GL89,
  title={Knots are determined by their complements},
  author={McA, Cameron and Luecke, John},
  journal={Journal of the American Mathematical Society},
  pages={371--415},
  year={1989},
  publisher={JSTOR}
}

@article{Tag23,
  title={Remarks on the minimalities of two-bridge knots in the ribbon concordance poset},
  author={Tagami, Keiji},
  journal={Bull. Belg. Math. Soc. Simon Stevin},
  volume={30},
  pages={317--327},
  year={2023}
}

@article{How81,
  author  = {Howie, James},
  title   = {On pairs of 2-complexes and systems of equations over groups},
  journal = {Journal f{\"u}r die reine und angewandte Mathematik},
  volume  = {324},
  year    = {1981},
  pages   = {165--174},
  doi     = {10.1515/crll.1981.324.165}
}

@article{Lac16,
  title={Elementary knot theory},
  author={Lackenby, Marc},
  journal={arXiv preprint arXiv:1604.03778},
  year={2016}
}

@article{CL23,
  title={Fibered and strongly quasi-positive {L}-space links},
  author={Cavallo, Alberto and Liu, Beibei},
  journal={Michigan Mathematical Journal},
  volume={73},
  number={1},
  pages={209--224},
  year={2023},
  publisher={University of Michigan, Department of Mathematics}
}

@article{OS04a,
  title={Holomorphic disks and genus bounds},
  author={Ozsv{\'a}th, Peter and Szab{\'o}, Zolt{\'a}n},
  journal={Geometry \& Topology},
  volume={8},
  number={1},
  pages={311--334},
  year={2004},
  publisher={Mathematical Sciences Publishers}
}

@article{OS08a,
  title={Holomorphic disks, link invariants and the multi-variable {A}lexander polynomial},
  author={Ozsv{\'a}th, Peter and Szab{\'o}, Zolt{\'a}n},
  journal={Algebraic \& Geometric Topology},
  volume={8},
  number={2},
  pages={615--692},
  year={2008},
  publisher={Mathematical Sciences Publishers}
}

@article{Gab86,
  title={Detecting fibred links in S 3},
  author={Gabai, David},
  journal={Commentarii Mathematici Helvetici},
  volume={61},
  number={1},
  pages={519--555},
  year={1986},
  publisher={Springer}
}

@incollection{Mil68,
  author    = {Milnor, John W.},
  title     = {Infinite cyclic coverings},
  booktitle = {Conference on the Topology of Manifolds},
  editor    = {Hocking, J. G.},
  pages     = {115--133},
  publisher = {Prindle, Weber \& Schmidt},
  address   = {Boston, MA},
  year      = {1968},
  note      = {Conference held at Michigan State University, East Lansing, Michigan, 1967}
}

@article{Ni06,
  title={A note on knot {F}loer homology of links},
  author={Ni, Yi},
  journal={Geometry \& Topology},
  volume={10},
  number={2},
  pages={695--713},
  year={2006},
  publisher={Mathematical Sciences Publishers}
}

@article{AT24,
  title={A generalization of the slice-ribbon conjecture for two-bridge knots and ${t_n}$-move},
  author={Abe, Tetsuya and Tagami, Keiji},
  journal={Tohoku Mathematical Journal},
  volume={76},
  number={4},
  year={2024},
  publisher={Mathematical Institute, Tohoku University}
}

@article{Lobb26,
  title={Khovanov concordance minima and the (4, 5) torus knot},
  author={Lobb, Andrew},
  journal={arXiv preprint arXiv:2602.12692},
  year={2026}
}

@article{Sun26,
  title={Twisted {A}lexander Polynomials, Ribbon Homology Cobordisms, and the {T}hurston Norm},
  author={Sun, Brian},
  journal={arXiv preprint arXiv:2604.20785},
  year={2026}
}

@article{Ba16,
  title={A note on the concordance of fibered knots},
  author={Baker, Kenneth L},
  journal={Journal of Topology},
  volume={9},
  number={1},
  pages={1--4},
  year={2016},
  publisher={Wiley Online Library}
}

@article{HP25,
  title={Ribbon knots and iterated cables of fibered knots},
  author={Hom, Jennifer and Park, JungHwan},
  journal={arXiv preprint arXiv:2507.20455},
  year={2025}
}

@article{Ni09,
  title={Link {F}loer homology detects the {T}hurston norm},
  author={Ni, Yi},
  journal={Geometry \& Topology},
  volume={13},
  number={5},
  pages={2991--3019},
  year={2009},
  publisher={Mathematical Sciences Publishers}
}

@article{Sta65,
  title={Homology and central series of groups},
  author={Stallings, John},
  journal={Journal of Algebra},
  volume={2},
  number={2},
  pages={170--181},
  year={1965},
  publisher={Elsevier}
}

@article{Juh16,
  title={Cobordisms of sutured manifolds and the functoriality of link Floer homology},
  author={Juh{\'a}sz, Andr{\'a}s},
  journal={Advances in Mathematics},
  volume={299},
  pages={940--1038},
  year={2016},
  publisher={Elsevier}
}

@book{JTZ21,
  title={Naturality and mapping class groups in Heegaard Floer homology},
  author={Juh{\'a}sz, Andr{\'a}s and Thurston, Dylan and Zemke, Ian},
  volume={273},
  number={1338},
  year={2021},
  publisher={American Mathematical Society}
}

@book{Ras03,
  title={Floer homology and knot complements},
  author={Rasmussen, Jacob Andrew},
  year={2003},
  publisher={Harvard University}
}

@article{OS04b,
  title={Holomorphic disks and knot invariants},
  author={Ozsv{\'a}th, Peter and Szab{\'o}, Zolt{\'a}n},
  journal={Advances in Mathematics},
  volume={186},
  number={1},
  pages={58--116},
  year={2004},
  publisher={Elsevier}
}

@article{OS08b,
  title={Link {F}loer homology and the {T}hurston norm},
  author={Ozsv{\'a}th, Peter and Szab{\'o}, Zolt{\'a}n},
  journal={Journal of the American Mathematical Society},
  volume={21},
  number={3},
  pages={671--709},
  year={2008}
}

@article{Sil92,
  title={On knot-like groups and ribbon concordance},
  author={Silver, DS},
  journal={Journal of pure and applied algebra},
  volume={82},
  number={1},
  pages={99--105},
  year={1992},
  publisher={Elsevier}
}

@article{Miy18,
  title={A note on genera of band sums that are fibered},
  author={Miyazaki, Katura},
  journal={Journal of Knot Theory and Its Ramifications},
  volume={27},
  number={12},
  pages={1871002},
  year={2018},
  publisher={World Scientific}
}

@article{BM24,
  title={Knot {F}loer homology, link {F}loer homology and link detection},
  author={Binns, Fraser and Martin, Gage},
  journal={Algebraic \& Geometric Topology},
  volume={24},
  number={1},
  pages={159--181},
  year={2024},
  publisher={Mathematical Sciences Publishers}
}

@article{Lis07,
  title={Sums of lens spaces bounding rational balls},
  author={Lisca, Paolo},
  journal={Algebraic \& Geometric Topology},
  volume={7},
  number={4},
  pages={2141--2164},
  year={2007},
  publisher={Mathematical Sciences Publishers}
}

@article{BD24,
  title={{F}loer homology, clasp-braids and detection results},
  author={Binns, Fraser and Dey, Subhankar},
  journal={arXiv preprint arXiv:2405.11224},
  year={2024}
}

@article{MM25,
  title={Slice disks modulo local knotting},
  author={Meier, Jeffrey and Miller, Allison N},
  journal={arXiv preprint arXiv:2503.09870},
  year={2025}
}

@article{Binns25,
  title={Closures of 3-braids and detection},
  author={Binns, Fraser},
  journal={Pacific Journal of Mathematics},
  volume={340},
  number={1},
  pages={1--36},
  year={2025},
  publisher={Mathematical Sciences Publishers}
}

@article{Hem87,
  title={Residual finiteness for 3-manifolds},
  author={Hempel, John},
  journal={Combinatorial group theory and topology (Alta, Utah, 1984)},
  volume={111},
  pages={379--396},
  year={1987}
}

@article{Pic20,
  title={The {C}onway knot is not slice},
  author={Piccirillo, Lisa},
  journal={Annals of Mathematics},
  volume={191},
  number={2},
  pages={581--591},
  year={2020},
  publisher={Department of Mathematics, Princeton University Princeton, New Jersey, USA}
}

@article{GL22,
  title={Khovanov homology and cobordisms between split links},
  author={Gujral, Onkar Singh and Levine, Adam Simon},
  journal={Journal of Topology},
  volume={15},
  number={3},
  pages={973--1016},
  year={2022},
  publisher={Wiley Online Library}
}

@article{Tor53,
  author  = {Torres, Guillermo},
  title   = {On the {A}lexander Polynomial},
  journal = {Annals of Mathematics},
  series  = {Second Series},
  volume  = {57},
  number  = {1},
  pages   = {57--89},
  year    = {1953},
  doi     = {10.2307/1969662}
}

@article{Wal68,
  title={On irreducible 3-manifolds which are sufficiently large},
  author={Waldhausen, Friedhelm},
  journal={Annals of Mathematics},
  volume={87},
  number={1},
  pages={56--88},
  year={1968},
  publisher={JSTOR}
}

@article{DLS26,
  title={Local equivalence and refinements of {R}asmussen's s-invariant},
  author={Dunfield, Nathan M and Lipshitz, Robert and Sch{\"u}tz, Dirk},
  journal={Journal of Topology},
  volume={19},
  number={1},
  pages={e70056},
  year={2026},
  publisher={Wiley Online Library}
}

@article{Bon25,
  title={Positive knots and ribbon concordance},
  author={Boninger, Joe},
  journal={Pacific Journal of Mathematics},
  volume={335},
  number={1},
  pages={81--95},
  year={2025},
  publisher={Mathematical Sciences Publishers}
}

@article{Ghi08,
  title={Knot {F}loer homology detects genus-one fibred knots},
  author={Ghiggini, Paolo},
  journal={American journal of mathematics},
  volume={130},
  number={5},
  pages={1151--1169},
  year={2008},
  publisher={Johns Hopkins University Press}
}

@article{FBW24,
  title={Fixed-point-free pseudo-{A}nosov homeomorphisms, knot {F}loer homology and the cinquefoil},
  author={Farber, Ethan and Reinoso, Braeden and Wang, Luya},
  journal={Geometry \& Topology},
  volume={28},
  number={9},
  pages={4337--4381},
  year={2024},
  publisher={Mathematical Sciences Publishers}
}

@article{Rud93,
  title={Quasipositivity as an obstruction to sliceness},
  author={Rudolph, Lee},
  journal={Bulletin of the American Mathematical Society},
  volume={29},
  number={1},
  pages={51--59},
  year={1993}
}

@article{Gil84,
  title={Ribbon concordance and a partial order on S-equivalence classes},
  author={Gilmer, Patrick M},
  journal={Topology and its Applications},
  volume={18},
  number={2-3},
  pages={313--324},
  year={1984},
  publisher={North-Holland}
}

@misc{KnotJob,
  author       = {Dirk Sch{\"u}tz},
  title        = {KnotJob},
  howpublished = {Computer program},
  note         = {\url{https://www.maths.dur.ac.uk/users/dirk.schuetz/knotjob.html}},
  year         = {2025}
}

@article{BD22,
  title={Cable links, annuli and sutured {F}loer homology},
  author={Binns, Fraser and Dey, Subhankar},
  journal={arXiv preprint arXiv:2207.08035},
  year={2022}
}

@article{HO08,
  title={The {O}zsv{\'a}th-{S}zab{\'o} and {R}asmussen concordance invariants are not equal},
  author={Hedden, Matthew and Ording, Philip},
  journal={American journal of mathematics},
  volume={130},
  number={2},
  pages={441--453},
  year={2008},
  publisher={Johns Hopkins University Press}
}

@inproceedings{BS25a,
  title={{F}loer homology and non-fibered knot detection},
  author={Baldwin, John A and Sivek, Steven},
  booktitle={Forum of Mathematics, Pi},
  volume={13},
  pages={e1},
  year={2025},
  organization={Cambridge University Press}
}

@article{Agol22,
  title={Ribbon concordance of knots is a partial ordering},
  author={Agol, Ian},
  journal={Communications of the American Mathematical Society},
  volume={2},
  number={09},
  pages={374--379},
  year={2022}
}

@article{LZ19,
  title={Khovanov homology and ribbon concordances},
  author={Levine, Adam Simon and Zemke, Ian},
  journal={Bulletin of the London Mathematical Society},
  volume={51},
  number={6},
  pages={1099--1103},
  year={2019},
  publisher={Wiley Online Library}
}

@article{Cha06,
  title={The effect of mutation on link concordance, 3-manifolds, and the Milnor invariants},
  author={Cha, Jae Choon},
  journal={Journal of Knot Theory and Its Ramifications},
  volume={15},
  number={02},
  pages={239--257},
  year={2006},
  publisher={World Scientific}
}

@article{Zem19a,
  title={Knot {F}loer homology obstructs ribbon concordance},
  author={Zemke, Ian},
  journal={Annals of Mathematics},
  volume={190},
  number={3},
  pages={931--947},
  year={2019},
  publisher={JSTOR}
}

@article{KSS82,
  author  = {Kawauchi, Akio and Shibuya, Tetsuo and Suzuki, Shin'ichi},
  title   = {Descriptions on Surfaces in Four-Space. {I}: Normal Forms},
  journal = {Mathematical Seminar Notes, Kobe University},
  volume  = {10},
  number  = {1},
  pages   = {75--125},
  year    = {1982}
}

@article{Mil54,
  author  = {Milnor, John},
  title   = {Link Groups},
  journal = {Annals of Mathematics},
  volume  = {59},
  number  = {2},
  pages   = {177--195},
  year    = {1954},
  doi     = {10.2307/1969685}
}

@article{HKM20,
  title={Isotopies of surfaces in 4--manifolds via banded unlink diagrams},
  author={Hughes, Mark C and Kim, Seungwon and Miller, Maggie},
  journal={Geometry \& Topology},
  volume={24},
  number={3},
  pages={1519--1569},
  year={2020},
  publisher={Mathematical Sciences Publishers}
}

@article{CS93,
  title={Reidemeister moves for surface isotopies and their interpretation as moves to movies},
  author={Carter, J Scott and Saito, Masahico},
  journal={Journal of Knot Theory and its Ramifications},
  volume={2},
  number={03},
  pages={251--284},
  year={1993},
  publisher={World Scientific}
}

@article{Zem19b,
  title={Link cobordisms and absolute gradings on link {F}loer homology},
  author={Zemke, Ian},
  journal={Quantum Topology},
  volume={10},
  number={2},
  pages={207--323},
  year={2019}
}

@article{BG24,
  title={Special Alternating Knots Are Band Prime},
  author={Boninger, Joe and Evan Greene, Joshua},
  journal={International Mathematics Research Notices},
  volume={2024},
  number={10},
  pages={8758--8763},
  year={2024},
  publisher={Oxford University Press}
}

@article{BS22,
  title={{L}-space knots are fibered and strongly quasipositive},
  author={Baldwin, John A and Sivek, Steven},
  journal={Open Book Series},
  volume={5},
  number={1},
  pages={81--94},
  year={2022},
  publisher={Mathematical Sciences Publishers}
}

\end{document}